\documentclass[12pt]{amsart}

\usepackage{amsmath,amsthm,amssymb,amsfonts}
\usepackage{units}
\usepackage{tikzsymbols}
\usepackage{mathabx}
\usepackage{etoolbox} 
\usepackage{bbm}
\usepackage[all,tips]{xy}
\usepackage{hyperref}

\newtheorem{thm}{Theorem}[section]

\newtheorem{lem}[thm]{Lemma}
\newtheorem{claim}[thm]{Claim}
\newtheorem{obs}[thm]{Observation}
\newtheorem{cor}[thm]{Corollary}
\newtheorem{prop}[thm]{Proposition}

\newtheorem{conj}[thm]{Conjecture}

\theoremstyle{definition}
\newtheorem{de}[thm]{Definition}
\newtheorem{ex}[thm]{Example}

\theoremstyle{remark}
\newtheorem{remark}[thm]{Remark}

\numberwithin{equation}{section}

\makeatletter

\newcommand{\Rmnum}[1]{\expandafter\@slowromancap\romannumeral #1@}
\makeatother

\newcommand{\bd}{\mathbf{d}}

\newcommand{\RR}{\mathbb{R}}
\newcommand{\NN}{\mathbb{N}}
\newcommand{\GG}{\mathbb{G}}
\newcommand{\SL}{\operatorname{SL}}
\newcommand{\GL}{\operatorname{GL}}
\newcommand{\Hom}{\operatorname{Hom}}
\newcommand{\Gal}{\operatorname{Gal}}
\newcommand{\Aut}{\operatorname{Aut}}
\newcommand{\res}{\operatorname{Res}}

\newcommand{\ZZ}{\mathbb{Z}}

\newcommand{\QQ}{\mathbb{Q}}
\newcommand{\CC}{\mathbb{C}}

\DeclareMathOperator{\diag}{\mathrm{diag}}

\newcommand{\Ad}{\operatorname{Ad}}

\DeclareMathOperator{\conjj}{{\mathrm{conj}}}

\DeclareMathOperator{\stab}{\mathrm{stab}}

\newcommand{\bra}{\left\langle}
\newcommand{\ket}{\right\rangle}
\newcommand{\acts}{\lefttorightarrow}

\newcommand{\conv}{{\rm conv}}

\newcommand{\fa}{\mathfrak{a}}

\newcommand{\fg}{\mathfrak{g}}

\newcommand{\fk}{\mathfrak{k}}

\newcommand{\fn}{\mathfrak{n}}

\newcommand{\fp}{\mathfrak{p}}

\newcommand{\norm}[1]{\left\|#1\right\|}

\def\gl{\operatorname{GL}}
\def\cat{\operatorname{CAT}(0)}

\renewcommand{\Im}{\operatorname{Im}}

\newcommand{\ta}{\mathtt{a}}
\newcommand{\tp}{\mathtt{p}}

\usepackage{enumitem}
\newlist{steps}{enumerate}{1}
\setlist[steps, 1]{label = Step \arabic*:}

\begin{document}

\title{Geometric interpretation of quantitative instability}
\author{Omri N. Solan}
\address{Einstein Institute of Mathematics, Hebrew University, Israel}
\email{omrinisan.solan@mail.huji.ac.il}
\author{Nattalie Tamam}
\address{Department of Mathematics, University of Michigan, Ann Arbor, Michigan, United States 48109}
\email{nattalie@umich.edu}
\date{}

\begin{abstract}
Given a real algebraic group action on a linear space $G\acts V$, a vector $v\in V$ is called unstable if $0\in \overline{Gv}-Gv$, where the closure is taken with respect to the Zariski topology. 
A fundamental theorem of Kempf \cite{kempf} in geometric invariant theory states that $v$ is unstable if and only if there is a one-parameter subgroup $A$ of $G$ such that $v$ is unstable with respect to it, i.e., $0\in \overline{Av}-Av$. 
Assuming $G$ is a semisimple real algebraic group defined over $\QQ$, we give a new proof to this result using a geometric interpretation of the setting.
In the process, we also give a new proof of an effective version of this result by Shah and Yang \cite[Prop. 2.2]{NimishPengyu}. 
Our interpretation involves relating the length of vectors under a linear action to convex functions on certain $\cat$-spaces, and bound the latter from below by Busemann functions.
\end{abstract}

\maketitle
\markright{}
\section{Introduction} 
\label{sec:introduction}
Geometric invariant theory studies the quotient of a variety $V$ by a real algebraic group action $G\acts V$. An important phenomenon occurs when an orbit closure of a point $v$ in $V$ contains a point which is not in the $v$-orbit, i.e, when $\overline{Gv}-Gv\neq\emptyset$, where the the topology considered is the Zariski topology. 
It is then a natural question to study the behavior of the orbit around such point. 
In \cite{kempf} Kempf studied the following question: 
\begin{quote} \centering 
    Given $v\in V$, $u\in \overline{Gv}-Gv$, for which curves $\{\gamma(t)\}$ in $G$ does the distance $\|u-\gamma(t)v\|$ shrink `fastest'?
\end{quote}
The phenomenon of non-closed $G$-orbits is called instability. In homogeneous dynamic the $\{0\}$-instability, i.e., the case $u$ is the zero vector in a $G$-representation $V$, is widely used when studying how far an orbit is into a cusp, see \cite{Mahler, escape mass, shah, Shapira}. 
Recently, Kempf's result has been used in the study of orbits and more specifically orbit closures in homogeneous dynamics. First by Yang \cite{Pengyu}, which used it to prove not only nondivergence, but also equidistribution of certain measures. Yang's usage of \cite{kempf} is extending Shah's basic lemma \cite[Prop. 4.2]{shah}, which was proved independently of Kempf's result. In \cite{KDSY, NimishPengyu} the method was pushed further, and a quantitative version of Kempf's result was proved, answering a seemingly different question:
\begin{quote} \centering 
    Assuming $0\in \overline{Gv}$, can the size $\norm{gv}$ for all $g\in G$ be controlled?
\end{quote}
The solution they provide is using highest weight representations: 
\begin{thm}[{\cite[Prop. 2.2]{NimishPengyu}}]\label{thm: NimishPengyu}
Assume $G$ is semisimple and let $\varrho:G\to \GL(V)$ be a $\QQ$-representation and $0\neq v\in V(\QQ)$ be an unstable vector, i.e., $0\in \overline{Gv}-Gv$. Then, there exists a $\QQ$-highest weight representation $\tilde\varrho:G\rightarrow\operatorname{GL}(W)$, a highest weight vector $w\in W(\QQ)$, and $a,c>0$ such that such that for all $g\in G$ 
\begin{align}\label{ineq: more then repfunc2}
\log \|\varrho(g)v\| \ge a \log \|\tilde\varrho(g)w\| - c.
\end{align}
\end{thm}

\begin{remark}
    Note that the highest weight in Theorem \ref{thm: NimishPengyu} is chosen with respect to some choice a maximal $\QQ$ split torus and a choice of $\Delta_\QQ$ (see \S\ref{sec:fundamental representation}).
\end{remark}

This theorem is important for their setting, but it has another application previously unknown, as it answers a question of Weiss, \cite[Question 3.4]{Weiss} positively. 
In this paper we will prove Theorem \ref{thm: NimishPengyu} using geometric tools, and give a geometric interpretation to the right hand side of Eq. \eqref{ineq: more then repfunc2}.

\subsection{Main results}
\label{ssec:main results}
\subsubsection{Geometric interpretation}
For the \textbf{left hand side of Eq. \eqref{ineq: more then repfunc2}},
we analyze the symmetric space of non-positive curvature $M:=K\backslash G$, where $K$ is a maximal compact subgroup of $G$.
For every real representation $\varrho:G\rightarrow \operatorname{GL}(V)$, with $V$ equipped with a $K$-invariant norm, the function $g\mapsto \log \|\varrho(g)v\|$ descends to a function $f_v:M\to \RR$. With the right choice of norm, $f_v$ is convex, see Observation \ref{obs: convexity of f_v} and Lemma \ref{lem: construction of bilinear form} for the existence of such a norm.
The \textbf{right hand side of Eq. \eqref{ineq: more then repfunc2}} can also be viewed as a multiple of $f_w$. However, since $w\in W$ is a highest weight vector, it has an intrinsic description, $f_w$ is a constant multiple of a Busemann function on $M$ (See Definition \ref{defn: Busemann function}). 
We use geometric tools to show the following: 

\begin{thm} \label{thm: main}
Let $\varrho:G\rightarrow V$ be a $\RR$-representation and assume a $K$-invariant norm on $V$. Then, for every $v\in V$ there exists a Busemann function $\beta$ and constants $a,c>0$ so that we have the following: 
\begin{enumerate}[label=\emph{(\arabic*)}, ref=\arabic*]
    \item \label{part: main ge busemann}
    Any $x\in M$ satisfies \[
    f_v(x)\ge a\beta(x)-c.\]
    \item \label{part: main eq busemann}
    If $v\in V$ is a highest weight vector, then for any $x\in M$ \[
    f_v(x)= a\beta(x)-c.\]  
\end{enumerate} 
\end{thm}

\begin{remark}
In this manuscript we only deal with the real and rational setting (see \S\ref{sec: relation to kempf} for the rational setting), but similar results should hold for a general field, using more general theory, see \cite{tits,borel}.
\end{remark}

\subsubsection{Algebraic properties}
Theorem \ref{thm: main} lacks the rational nature of Theorem \ref{thm: NimishPengyu}. To this end we give various algebraic descriptions for `rational' Busemann functions.
The equivalent descriptions for Busemann functions can be found at Theorem \ref{thm: class equiv busemann} and the equivalent description for `rational' Busemann functions can be found in Theorem \ref{thm: class equiv busemann alg}.
Each of these theorems has some equivalent description, and in particular contains the second part of Theorem \ref{thm: main}. The first is the standard definition of the Busemann function, in Theorem \ref{thm: class equiv busemann} and a restriction of the geodesic in Theorem \ref{thm: class equiv busemann alg}. The second uses homomorphisms from a parabolic subgroup. The third uses a linear combination of functions of the form $f_v$, for $v$ in a special collection of representations, called fundamental representations, (see Definition \ref{def: fundumental weights}). 
Theorem \ref{thm: class equiv busemann alg} has a forth equivalent description, which is the right hand side of Eq. \eqref{ineq: more then repfunc2}.

\subsection{On the proof}
\hfill\break

\textbf{Analyzing Busemann functions:}
In \S\ref{sec:Busemann functions} we show the connection between the Busemann functions and the fundamental representations. In particular, we prove Theorem \ref{thm: class equiv busemann} which implies Theorem \ref{thm: main} Part \eqref{part: main eq busemann}.  

\textbf{Convex geometry:}
In \S\ref{sec:The fastest shrinking geodesic}-\ref{sec:shrink-rate function} we prove Theorem \ref{thm: busemann bound} which implies Theorem \ref{thm: main} Part \eqref{part: main ge busemann}.
We investigate the geodesic paths in $G$ that `shrink' $v$ the most, and prove Theorem \ref{thm: main} Part \eqref{part: main ge busemann}. It has two key steps. In \S\ref{sec:The fastest shrinking geodesic} we find a `fastest shrinking geodesic' for all `nice' convex functions. In particular, we are able to find the geodesic that shrinks a given vector in a representation the fastest.
This part is a geometric analogous to Kempf's result \cite{kempf}.

In \S\ref{sec:shrink-rate function} we show that the function defined by the left hand side of \eqref{ineq: more then repfunc2} is indeed `nice', and use some of its algebraic properties to show that it can be bounded below by a (specific) Busemann function. Theorem \ref{thm: busemann bound} is the main theorem of this part.
This part provides a geometric way to view \cite{kempf}, and a good intuition for Theorem \ref{thm: busemann bound}. 

\textbf{Algebraic view point:}
In the previous part we have constructed a geometric object, the Busemann function, (which is also algebraic by Theorem \ref{thm: class equiv busemann alg}) given algebraic data, a vector in an algebraic representation. Next, we wish to say that since the vector is defined over $\QQ$, it is invariant under $\Gal(\CC/\QQ)$, hence so does the Busemann function, i.e., it is also defined over $\QQ$. This requires us to find an algebraic description for the fastest shrinking geodesic of a vector. Such an algebraic description was defined and proved by Kempf in \cite{kempf}. In \S\ref{sec: relation to kempf} we relate our constructions to Kempf's construction, and use his result to deduce that our `fastest shrinking geodesic' is indeed algebraic. Thus, the inequality in \eqref{ineq: more then repfunc2} is satisfied for the rational fundamental representations (and not only the real ones).

In \S\ref{sec: proof of main theorem} we give an alternative proof of Theorem \ref{thm: NimishPengyu} using a more precise version of Theorem \ref{thm: main}.

\subsection{Further research}
\label{ssec:further research}
Theorem \ref{thm: busemann bound} gives a lower bound using a Busemann function for every convex function on a symmetric space of noncompact type $f:M\to \RR$ coming from a vector in a representation.  This yields the following conjecture:
\begin{conj}
For every non-constant convex function $f$ on $M$, there is a Busemann function $\beta:M\to \RR$, $a>0, C\in \RR$ such that for every $x\in M$ \[f(x)>a\beta(x)+C.\] 
\end{conj}
The anlagous conjecture for trees is false:
Let $G$ be a leafless $d$-regular tree, $x_0\in G$ a vertex and $(x_i)_{i=0}^\infty$ an infinite ray. 
Let $f:G\to \RR$ be the function defined as follows: $f(x_i) = -\sqrt i$ for $i\ge 0$, and for other $x\in G$ let $x_{i_0}$ be the closest element to $x$ on the ray. 
Define $f(x):= -\sqrt {i_0} + d(x, x_{i_0}) (\sqrt {i+1} - \sqrt i)$.
One can verify that it is convex and not bounded from below by any multiple of a Busemann function. 

Failures to extend this example to symmetric spaces leads us to believe the conjecture.

\subsection*{Acknowledgment}
We would like to thank B. Weiss for suggesting this problem and for familiarizing us with \cite{kempf}. We would also like to thank P. Yang for bringing our attention to the new results in this topic.
The first author would like to thank N. Tur, O. Bojan, and A. Levit for helpful discussions. The second author would also like to thank R. Spatzier for his interest in this project and for several helpful discussions and comments.

\section{Preliminaries} 
\label{sec:preliminaries}
\subsection{Homomorphisms} We will use several types of homomorphisms, Algebraic ones and topological ones. 
The set of continuous homomorphisms between topological groups will be denoted by $\Hom$. 
The set of $\ell$-algebraic homomorphisms between $\ell$-algebraic groups will be denoted by $\Hom_\ell$.






\subsection{Hadamard spaces}\label{sec: cat0 spaces}
A \emph{Hadamard space} is defined to be a nonempty complete metric space $(M,d)$ such that, given any points $x,y\in M$, there exists $m\in M$ such that for every $z\in M$ we have 
\begin{equation}\label{eq: Hadamard inequality}
    d(z,m)^{2}+{\frac{d(x,y)^{2}}{4}}\leq {\frac{d(z,x)^{2}+d(z,y)^{2}}{2}}.
\end{equation}
The point $m$ is called the \emph{midpoint} of $x$ and $y$, and it satisfies $d(x,m)=d(y,m)=d(x,y)/2$. we study the behavior of convex functions in such spaces. The main properties of them which are of use to us are listed in Lemma \ref{lem: cat0 implication} below. 

Alternatively, a space is Hadamard if it is a complete \emph{$\cat$-space}.
A metric space $(M,d)$ is a \emph{$\operatorname{CAT}(0)$-space} if it is geodesic (as defined below) and every `small enough' geodesic triangle satisfies a certain inequality. Such spaces were first defined and studied by Gromov, see \cite{gromov}.  
For a more detailed discussion on complete and $\cat$-spaces see \cite{Metric Spaces}. 

\begin{de}[Geodesic]
Let $(M, d)$ be a metric space. 
Given $x,y\in M$, a \emph{geodesic from $x$ to $y$} is a map $\gamma:I\rightarrow Y$, where $I=[a,b]\subset\RR$ is a closed interval, such that $\gamma(a) = x$, $\gamma(b) = y$ and \[d(\gamma(s), \gamma(s'))=|s-s'| \]
for all $s, s'\in I$ (in
particular, $d(x, y)=b-a$).
By abuse of notations, we identify geodesics with their images.
A \emph{geodesic ray} in $M$ is a map $\gamma: [0,\infty)\rightarrow M$ such that $d(\gamma(s), \gamma(s'))=|s-s'|$ for all $s,s'\ge 0$. 
$(M, d)$ is said to be a \emph{geodesic space} if every two points in $M$ are joined by a geodesic.
\end{de}

\begin{lem}\label{lem: cat0 implication} 
Let $(M,d)$ be a Hadamard space. 
Let $\gamma_1, \gamma_2:[0,\infty)\to M$ be geodesic rays such that $\gamma_1(0) = \gamma_2(0)$.
Then:
\begin{enumerate}
    \item For every $s>t>0$ we have $$\frac{d(\gamma_1(s), \gamma_2(s))}{s} \ge \frac{d(\gamma_1(t), \gamma_2(t))}{t}.$$ 
    \item In particular, for $r:= d(\gamma_1(1), \gamma_2(1))$ and all $t\ge 1$, we have \[d(\gamma_1(t), \gamma_2(t)) \ge rt.\]
    \item For every $t\ge 1$, the midpoint $m_t$ of $\gamma_1(t)$ and $\gamma_2(t)$ satisfies \[
    d(\gamma_1(0), m_t) \le \sqrt{1-r^2/4}\cdot t.\]
\end{enumerate}
\end{lem}

\begin{proof}
According to \cite[Prop. 2.2]{Metric Spaces} the distance function in a $\cat$-space is convex, i.e., for any $s>t>0$ 
\begin{align*}
    d(\gamma_1(t),\gamma_2(t))&\le \frac{t}{s}d(\gamma_1(s),\gamma_2(s))+\left(1-\frac{t}{s}\right)d(\gamma_1(0),\gamma_2(0)).
\end{align*}
Since we assume $\gamma_1(0)=\gamma_2(0)$, Claim (1) follows. 
Claim (2) follows directly from Claim (1).
By \eqref{eq: Hadamard inequality} and Claim (2) we have
\begin{align*}
    d(\gamma_1(0),m_t)&\le \sqrt{ d(\gamma_1(0),\gamma_1(t))^2/2+d(\gamma_2(0),\gamma_2(t))^2/2- d(\gamma_1(t),\gamma_2(t))^2/4}\\
    &\le\sqrt{ t^2- (rt/2)^2},
\end{align*}
proving Claim (3). 
\end{proof}

\begin{de}[Convex functions]
A function $f:M\to \RR$ is \emph{convex} if for 
every geodesic $\gamma:I\to M$ the composition 
$f\circ \gamma$ is a convex function on the interval $I$.
\end{de}

\subsection{Symmetric spaces of non-compact type}\label{sec:symmetric spaces}
There are many results on \emph{symmetric spaces of non-compact type}, i.e., Riemannian manifolds of non-positive sectional curvature, whose group of symmetries contains an inversion symmetry about every point. 
In particular, such spaces are Hadamard spaces. 
Here we present some of their geometric properties as well as some explicit constructions for future use.
\subsubsection{Overview}
Let $G$ be an $\RR$-reductive group, i.e., connected, linear, algebraic group over $\RR$ with a trivial unipotent radical, and $K$ be a maximal compact subgroup of $G$. 
Let $M := K\backslash G$, and $\pi:G\to M$ be the projection map.
We can define a $G$-invariant Riemannian metric $d_M$ on $M$ (See Definition \ref{de:sym-metric}), which give rise to a metric in the standard way.
Then, $(M,d_M)$ is a symmetric space of non-compact type.

We are also interested in a specific algerbraic subgroup, $A\subset G$, which is a maximal real split torus, i.e., $A\cong (\RR^\times)^r$, were $r$ is the real rank of $G$. The subgroup $A$ is called the \emph{Cartan torus} and has certain compatibility conditions with $K$, see Definition \ref{de: cartan decomposition}. Let $\fa\subset\fg$ denote the Lie algebra of $A$. 

The metric structure on $M$ has some phenomenal properties, which we now describe. 

\subsubsection{Metric properties}
Geodesics in symmetric spaces can be generalized into the higher dimensional concept of flats:
\begin{de}[Flats, geometric description]\label{de: flats}
A map $a:\RR^k\to M$ is called \emph{flat} if the pullback metric on $\RR^k$ is the Euclidean metric.
We abuse notations and do not distinguish between $a$ and its image $\Im(a)$.
A flat is called \emph{maximal} if no other flat properly contains it. In our setting, the dimension of all maximal flats of $M$ is equal to the $\RR$-rank of $G$.  
\end{de}

\begin{thm}[Algebraic description of flats, {\cite[Rem. 10.60(5)]{Metric Spaces}}]\label{thm: maximal flats}
For every Cartan torus $A$, all maximal flats can be described as translations of $\pi(A)$. 
In particular, for every flat $a:\RR^k\to M$ there is $g\in G$ and a continuous homomorphism $\iota:\RR^k\to A$ such that $a(x) = \pi(\iota(x)g)$.
\end{thm}

Theorem \ref{thm: maximal flats} also implies an explicit description of the geodesics in $M$, as can be seen in the next chapter. 


\begin{claim}\label{claim: compact acts transitively on flats}
The group $K$ acts transitively on maximal flats containing $\pi(e)$. Moreover,
for every geodesic $\gamma$ containing $\pi(e)$ the subgroup $\stab_{K}(\gamma)$ acts transitively on the set of maximal flats which contain $\gamma$.
\end{claim}
\begin{proof}
The group $G$ acts transitively on the maximal flats by \cite[Rem. 10.60]{Metric Spaces}. Moreover, since $A$ acts transitively on $\pi(A)$, $G$ acts transitively on pairs $(p,F)$ of a point $p$ in a maximal flat $F$. Since $K$ is the stabilizer of $\pi(e)$, we deduce that it acts transitively on the set of maximal flats containing $\pi(e)$. 

As for the second part, assume that $\gamma(t) = \pi(\exp(\ta t))$ for some $\ta\in \fa$. Note that 
\[
    G' := \operatorname{Stab}_G(\ta) = \{g\in G:Ad_g(\ta) = \ta\},
\]
is a reductive subgroup, as defined in \cite[Def. 10.56]{Metric Spaces}. We now remark that the first part of the proof, as well as the construction of the symmetric space in \cite[Thm. 10.58]{Metric Spaces}, follow for reductive subgroups, and not only semisimple groups. As defined in \cite[Thm. 10.58]{Metric Spaces}, the symmetric space $M'$ of $G'$ is a subspace of $M$, which contains all maximal flats of $M$ which contains $\gamma$. 
The second part of the claim now follows from the first part, applied to $M'$.
\end{proof}
The following corollary is an algebraic analog to Claim \ref{claim: compact acts transitively on flats}.
Recall that a Cartan algebra is a maximal Abelian subalgebra. 
\begin{cor}\label{cor: transitive on Cartan}
     The group $K$ acts transitively on the collection of Cartan algebras. For every $\ta\in \fg$ the stabilizer $\stab_K(\ta)$ acts transitively on the collection of Cartan Algebras containing $\ta$. \qed
\end{cor}



\subsubsection{Explicit construction}\label{sec: explicit construction}
We follow standard notation and results, see \cite[\S 2]{geometry} and \cite[\S IV]{helgarson}. 
Assume in addition that $G$ is semisimple, i.e., the Killing form on $\fg = \operatorname{Lie}(G)$ is nondegenerate. Let $K$ be a maximal compact subgroup of $G$. Then, $M: = K\backslash G$ is a manifold. 
Fixing $o=[K]\in M$ which is stabilized by $K$, define a projection $\pi: G\rightarrow M$ which is given by \[
\pi(g)=og,\] 
for any $g\in G$.  
Then, $G$ acts on $M$ by right multiplication. 


\begin{de}[Cartan decomposition of $\fg$]\label{de:Cartan decom algebra}
Recall that we fixed a maximal compact subgroup $K\subseteq G$. Denote its Lie algebra by $\fk = \operatorname{Lie}(K)$. The Killing form $B_\fg(\cdot,\cdot)$ is negative definite on $\fk$, and positive definite on its orthogonal complement $\fp = \fk^\perp$. 

The decomposition $\fg = \fk \oplus \fp$ is the \emph{Cartan decomposition of $\fg$}.
Note that the adjoint action of $K$ on $\fg$ preserves $\fk$ and $B_\fg$, and hence preserves $\fp$ as well.
\end{de}

\begin{de}[Metric on $M$]\label{de:sym-metric}
Fixing $o=[K]\in M$ which is stabilized by $K$, note that $T_o M = \fg/\fk\cong \fp$, thus define the positive definite bilinear form $B_o$ on $T_o M$ to be the the restriction of $B_\fg$ to $\fp$. Since $B$ is $K$ invariant, we may use the $G$ action on $M$ and define a Riemannian metric $B_p$ on $T_pM$ for every $p\in M$.
As usual in Riemannian geometry, 
for every curve segment $\gamma:[0,1]\rightarrow M$ we define the \emph{arc length} of $\gamma$ by
\begin{equation}\label{eq: arc length}
    L(\gamma)=\int_0^1 B_{\gamma(t)}(\dot\gamma(t),\dot\gamma(t))^{1/2}dt,
\end{equation}
and the metric 
\[d_M(p,q):= \inf\{L(\gamma)\ \ |\ \ \gamma:[0,1]\to M, \gamma(0) = p,\gamma(1)=q\}.\]
\end{de}
\begin{de}[Cartan decomposition of $G$]
\label{de: cartan decomposition}
Fix a maximal abelian $\fa\subseteq \fp$. Then $A = \exp \fa$ is a maximal split torus in $G$, and one may write (see for example \cite[Thm. 7.39]{knapp}) 
\begin{equation}
    G=KAK. 
\end{equation}
In this case $A$ is called \emph{Cartan torus}.
\end{de}

The next result follows directly from Theorem \ref{thm: maximal flats} (see also \cite[2.4.2]{geometry}). 
\begin{cor}\label{cor: geodesics description}
All geodesics starting at $\pi(e)$ are of the form $t\mapsto \pi(\exp(t\ta)g)$ for some $\ta\in \fa$ with $B(\mathtt a, \mathtt a) = 1$ and $g\in K$. Alternatively, such geodesic is of the form $t\mapsto \pi(\exp(t\mathtt{p}))$ for some $\mathtt{p}\in \fp$ with $B(\mathtt p, \mathtt p) = 1$. 
\end{cor}






\subsubsection{Example: $G=\operatorname{SL}_n(\RR)$}
\label{sssec: SLnR}
In this case $K=\operatorname{SO}(n)$ is a maximal compact subgroup of $G$, and $K\backslash G$ is isomorphic to the space of positive-definite symmetric $n\times n$-matrices, with determinant $1$, denoted by $P(n,\RR)$. The group $G$ acts on $P(n,\RR)$ by conjugation. Then, $K$ is the stabilizer of the identity matrix $I$. 

The Cartan decomposition is defined by $\fp,\fk$ being the symmetric and antisymmetric matrices in $\mathfrak{sl}_n=\operatorname{Lie}(G)$.
The tangent
space $T_pP(n,\RR)$ at a point $p\in P(n,\RR)$ is naturally isomorphic (via translation) to $\fp$, and the pseudo-Riemannian structure there is defined by \[
B_{p}(Y,Z)=\operatorname{Tr}(p^{-1}Yp^{-1}Z). \] 
The maximal $\RR$-split torus $A$ here is the set of diagonal matrices. 
For more information see \cite[{\S II.10}]{Metric Spaces}. 

\subsection{Representations of \texorpdfstring{$G$}{G} and short vectors} \label{sec:representations}
Fix a subfield $\ell\subseteq \RR$. We assume that $G$ is defined over $\ell$, and frequently use some of the algebraic structure. The reader should note that the Cartan decomposition is not necessarily defined over $\ell$. 
We use the following standard notation of arithmetic groups (see~\cite{borelf,borel}). Fix a subfield $\ell\subseteq \RR$ and a maximal $\ell$-split torus $A_\ell$ in $G$, and denote its Lie-algebra by $\fa_\ell$. 
We can conjugate $K$ and obtain that $\fa_\ell \subset \fp$. Note that $\fp$ is not necessarily defined over $\ell$. 

Given an $\ell$-representation $\varrho:G\rightarrow\gl(V)$, we denote by $\Phi_\varrho$ the set of \emph{$\ell$-weights} of $G$, i.e., the set of characters $\lambda\in \fa^*_\ell$ such that the subspace \[
V_\lambda=\left\{v\in V:\text{for all }a=\exp(\ta)\in A_\ell,\: \varrho(a)v=\exp(\lambda(\ta))v\right\} \]
is not trivial. The space $V_\lambda$ is called the \emph{weight space corresponding to $\lambda$}, and elements of $V_\lambda$ are called \emph{weight vectors corresponding to $\lambda$}.
Then, there is a decomposition 
\begin{equation}\label{eq: wheight decomposition}
    V=\bigoplus_{\lambda\in\Phi_\varrho}V_\lambda. 
\end{equation}

The following lemma defines a `nice' quadratic form on $V$ which we use to define a norm on it. It is proved in the Appendix. 

\begin{lem}[Construction of bilinear form]\label{lem: construction of bilinear form}
If $\ell = \RR$ and $A$ is a cartan torus, then there is a $K$-invariant positive bilinear form $\bra \cdot, \cdot\ket$ on $V$ so that the linear spaces $V_\lambda$ are orthogonal with respect to it. 
\end{lem}
Since $\fa_\ell \subset \fp$, we deduce that $A_\ell$ is contained in some Cartain torus, and hence we may apply Lemma \ref{lem: construction of bilinear form} also for $A_\ell$.

\subsection{Parabolic subgroups and their properties}\label{sec:parabolic subgroups}
We continue to use the notation of \S \ref{sec:representations}
We use standard notation and results about parabolic groups in symmetric spaces (see \cite[\S 11]{borel}, \cite[\S 2.17]{geometry}, or \cite[\S II.10]{Metric Spaces}). We also prove some properties of them to be used in later chapters. 

The classical, algebraic, definition of a \emph{parabolic group} is a closed subgroup $P$ of $G$ so that $G/P$ is a projective variety. 
In our view of $G$, as a group acting on a symmetric space, the following, more geometrical, definition of a parabolic group is more informative. See \cite[Cor. 11.2]{borel} and \cite[Prop. 2.6]{mumford}. 

\begin{de}[Parabolic and Unipotent subgroups]
A subgroup $P$ is called \emph{parabolic} if it is of the form \[
P=P_{\ta}:=\left\{g\in G:\lim_{t\rightarrow\infty}\exp(-t\ta)g\exp(t\ta)\text{ exists}\right\}, \]
where $\ta\in\fp$.
If $G$ is defined over $\ell$ then the group $P$ is called \emph{$\ell$-parabolic} if it is defined over $\ell$ as an algebraic group. It may be an $\ell$-parabolic even if $\ta$ is not $\ell$-algebraic.
There are many different elements $\ta$ defining the same $P_\ta$. 

A Borel subgroup of $G$ is a maximal closed, connected, solvable, subgroup of $G$. 
The \emph{unipotent radical} of a parabolic group $P$ is the unipotent part of the  intersection of all Borel subgroups which are contained in $P$. Explicitly, the unipotent radical of $P_{\ta}$ is \[
U_{\ta}=\left\{g\in G:\lim_{t\rightarrow\infty}\exp(-t\ta)g\exp(t\ta)=e\right\}. \]
The unipotent radical is always nilpotent, and if $P$ is an $\ell$ parabolic then the unipotent radical is also defined over $\ell$. Moreover, the unipotent radical depends only on $P_\ta$ and not on $\ta$.
\end{de}


The following subset of the semisimple part of the  Levi decomposition of $P_\ta$ is of special interest for us 
\begin{align}\label{eq: levi element of Pa}
    T_{\ta}&:=\exp\{\stab_\fg(\ta)\cap\fp\},
\end{align}
where $\stab_\fg(\ta)=\{\mathtt{b}\in\fg:[\ta,\mathtt{b}]=0\}$. In particular, $T_\ta\subset\stab_{G}(\ta)$. Note that $T_\ta$ is not a group.
Let $K_\ta:=K\cap P_\ta$. 

\begin{lem}[{Generalized Iwasawa decomposition, \cite[Prop. 2.17.5]{geometry}}]\label{lem: Iwasawa}
For any $\ta\in\fp$ we have $P_\ta=K_\ta\cdot T_\ta\cdot U_\ta$ and $G = K\cdot T_\ta\cdot U_\ta$. Moreover, in both equations the indicated decomposition is unique. It follows that $G = K P_{\ta}$.
\end{lem}
\begin{remark}
Our symbols $\ta, P_\ta, U_\ta, T_\ta, K_\ta$ correspond to the symbols $X, G_X, N_X, A_X, K_X$ in \cite{geometry}.
\end{remark}

\begin{claim}\label{claim: U fixes highest weight}
    Let $P = P_\ta$ for some $\ta\in \fp$. 
    Let $G\acts V$ be a representation of $G$ and $v\in V$ be a vector such that $\RR v$ is $P$-invariant. 
    Then $v$ is $U_\ta$-invaraint. 
\end{claim}
\begin{proof}
    Since $\exp(t\ta)\in P$ for every $t\in \RR$, we deduce that $\exp(t\ta)v = \lambda_t v$ for some $\lambda_t\in \RR$. 
    Hence, for every $u\in U_\ta$ we have \[uv = \exp(-t\ta)u\exp(t\ta)v \xrightarrow{t\to \infty }v.\]
\end{proof}
\begin{ex}[Example \ref{sssec: SLnR} continued]
    Assuming $G=\SL_n(\RR)$, every parabolic subgroup is the stabilizer of a flag $0<V_1<\cdots<V_{s-1} < V_s = \RR^n$ in the space of flags. 
    Thus, every parabolic subgroup is conjugated to a group consisting of all block upper triangular matrices.
    
    For example, when $n=3$, there are, up to conjugation, three proper parabolic subgroups:
    \[
    \begin{pmatrix}
        * & * & *\\ & *& * \\  & & * \end{pmatrix},\quad
    \begin{pmatrix}
        * & * & *\\ *& *& * \\  & & * \end{pmatrix},\quad
    \begin{pmatrix}
        * & * & *\\ & *& * \\  & *& *
    \end{pmatrix}, \]  where the leftmost one is a minimal parabolic.
\end{ex}

\subsection{The fundamental representation}\label{sec:fundamental representation}
In this section, we keep the setting and notation of \S \ref{sec:representations} and \S\ref{sec:parabolic subgroups}. We also follow standard notation and results (see \cite{fultonharris, knapp, BT}).

The Killing form defines an inner product on $\fa$ (and so also on $\fa^*$), which we denote by $\langle\cdot,\cdot\rangle$. 

\begin{de}[The Root system and Weyl group]\label{defn:root and weyl}
Let $\fa_\ell$ be a maximal $\ell$-split torus of $\fg$.  
We denote by $\Phi_\ell$ the root system of $\fa_\ell$, i.e. the set of non-trivial eigenvalues with respect to the adjoint action of $\fa_\ell$ on $\fg$, and by $W(\Phi_\ell)$ the Weyl group of $\Phi_\ell$, i.e. the group generated by the reflections $w_\lambda$, $\lambda\in\Phi_\ell$, defined by
\begin{equation}\label{eq: weyl defn}
    w_{\lambda}(\chi)= \chi - 2\frac{\langle\chi,\lambda\rangle}{\langle\lambda,\lambda\rangle}\lambda,
\end{equation}
for any characters $\chi\in\fa^*_\ell$. 
\end{de}

\begin{lem}[{\cite[Prop. 2.68]{knapp}}] \label{lem:Weyl so dominant}
For any $\chi\in\fa^*_\ell$ there exists an element $w\in W_\ell$ so that for every $\alpha\in\Delta_\ell$ \[
\langle w(\chi),\alpha\rangle\ge0.\]
\end{lem}

\begin{remark}\label{rem:Weyl in K}
    By \cite[\S 11.19]{borel}, the Weyl group can be realized as $$N_G(A_\ell)/Z_G(A_\ell),$$ where $N_G(A_\ell)$ is the normalizer of $A_\ell$ in $G$ and  $Z_G(A_\ell)$ is the centralizer of $A_\ell$ in $G$. 
    Moreover, if $\ell = \RR$ and $A_\ell$ is the Cartan torus, \cite[\S IV.6]{knapp} shows that the representatives of $W_\ell$ in $N_G(A_\ell)/Z_G(A_\ell)$ can be chosen to be from $K$. 
\end{remark}

Recall that we denote $\fa=\fa_\RR$.
For every $\RR$-Weyl chamber $\fa^\circ \subset\fa_\ell-\bigcap_{w\in W(\Phi_\RR)}\ker(w)$, and any $\ta\in\fa^\circ$, the group $U_{\ta}$ depends only on $\fa^\circ$ and not on $\ta$. It is denoted $N_{\fa^\circ}$. In particular, as a special case of Lemma \ref{lem: Iwasawa} we have the decomposition 
\begin{equation} \label{eq: Iwasawa}
G=K\cdot A\cdot N_{\fa^\circ}. 
\end{equation}


\begin{de}[The simple system and the highest weight]\label{def: fundumental weights} See \cite{BT} for the definitions and claim below.
Let $\Delta_\ell=\{\alpha_1,\dots,\alpha_r\}$ be an $\ell$-simple system of $\Phi_\ell$. 
Given an $\ell$-representation $\varrho:G\rightarrow\GL(V)$ and $\lambda\in\Phi_{\varrho}$, we say that $\lambda$ is the $\ell$-highest weight of $\varrho$ if for any $\lambda'\in\Phi_\varrho$ we have \[
\lambda-\lambda'\in\operatorname{span}_{\NN\cup\{0\}}\Delta_\ell. \] 
Note that a different choice of simple system yields a different highest weight.  
\end{de}

The next result follows from the construction in \cite[\S7]{borel3}. 
\begin{lem}
For any irreducible $\ell$-representation $\varrho$ there exists an $\ell$-highest weight. 
\end{lem}

\begin{lem}\label{lem:stabilized by parab}
Let $\varrho:G\rightarrow\GL(V)$ be an $\ell$-representation of $G$, $P$ be a parabolic subgroup of $G$ and $v\in V$ satisfy that $\varrho(P)v\subseteq\RR v$. Then, there exists an irreducible sub-representation $\varrho':G\rightarrow\GL(V')$ of $\varrho$ for some $V'<V$ so that $v$ is its $\ell$-highest weight vector with respect to some choice of a simple system. 
\end{lem}

\begin{proof}
    Let $V'$ be the subspace of $V$ spanned by $\varrho(G)v$ and $\varrho':G\rightarrow\GL(V')$ be the implied sub-representation of $\varrho$. If we show that $\varrho'$ is irreducible, then the claim will follow from \cite[\S7]{borel3}. 

    Let $G^\CC$ be the complexification of $G$, $\varrho^\CC:G^\CC\to GL(V^\CC)$ be the complexification of $\varrho$. It is enough to show that $\varrho^\CC(G^\CC).v$ generate an irreducible representation in $V^\CC$. 
    Let $B<P^\CC$ be a Borel subgroup of $G^\CC$. 
    Let $\sigma:B\to \CC^\times$ be the character of $B$ using its action on $\CC v$. 
    The construction of the Verma-module (see \cite[\S V.3]{knapp}) provides us with the unique irreducible representation of $G^\CC$ with a vector $v'$, on which $B$ acts by $\sigma$, and concludes the proof.
\end{proof}

Let us look at the adjoint representation $\Ad:G\rightarrow\operatorname{Aut}(\fg)$. As described in \S\ref{sec:representations}, it has the decomposition \[
\fg=\bigoplus_{\lambda\in\Phi_\ell=\Phi_{\Ad}}\fg_{\lambda}.\]

\begin{de}
For any $1\le i\le r$ let $P_i$ be the $\ell$-parabolic subgroup of $G$ with Lie algebra
\[\operatorname{Lie}(P_i) = \bigoplus_{\lambda\not\le -\alpha_i} \fg_\lambda\]
(see \cite[\S V.7]{knapp}).
\end{de}
\begin{remark}
    For any $1\le i\le r$ the set $\{\lambda\in \Phi_\ell: \lambda\not\le -\alpha_i\}$ is equal to the collection of $\lambda \in \Psi$ such that in the representation $\lambda = \sum_{j=1}^r c_j \alpha_j$ we have $c_i\ge0$. 
\end{remark}

For any $1\le i\le r$ let 
\[
\chi_i:=\sum_{\lambda\not \le -\alpha_i}\lambda\dim\fg_\lambda,\quad d_i:=\sum_{\lambda\not \le -\alpha_i}\dim\fg_{\lambda} = \dim \operatorname{Lie}(P_i).
\]
Let $\bigwedge^{d_i}\fg$
be the implied wedge representation such that
for any $g\in G$ and $\bigwedge_{j=1}^{d_i}v_j\in \bigwedge^{d_i}\fg$,
\[
\hat\varrho_i(g)\left(\bigwedge_{j=1}^{d_i}v_j\right) =\bigwedge_{j=1}^{d_i}\Ad(g)(v_j).\]
Let $(v_{j, i})_{j=1}^{d_i}$ be a basis of $\operatorname{Lie}(P_i)$. Then $v_i = \bigwedge_{j=1}^{d_i} v_{j, i}$ satisfies that $\RR v_i$ is $P_i$ invariant. The $G$-sub-representations $\varrho_i:G\to \GL(V_i)$ generated by $v_i$ is irreducible by Lemma \ref{lem:stabilized by parab}, and is termed \emph{fundamental representations}. The vector $v_i$ is of highest weight $\chi_i\in\Phi_{\varrho_i}$ in $V_i$.


The following result is known to experts, but as we did not find it in the literature, we prove it here. It shows that the representations constructed in this section are `almost' the $\ell$-fundamental representations, which are the ones who satisfy the conclusion of Proposition \ref{prop:fundamental rep} with minimal $m_i$. 

\begin{prop}\label{prop:fundamental rep}
For any $1\le i\le r$ the highest weight satisfies
\begin{equation}\label{eq: fundamental weight}
    \langle\alpha_j,\chi_i\rangle=m_{i}\delta_{i,j},
\end{equation}
where $m_i$ is a positive integer and $\delta_{i,j}$ is Kronecker delta, for all $1\le j\le r$. 
\end{prop}

\begin{proof}
In a similar fashion to \cite[Lemma 5.1]{tamam2}, for any $j\neq i$, $\chi_i$ is invariant under the action of $w_{\alpha_j}$ (see Definition \ref{defn:root and weyl}), which implies \eqref{eq: fundamental weight} when $j\neq i$. Since $\chi_i$ is a non-zero, non-negative integer combination of the $\ell$-simple system, \eqref{eq: fundamental weight} must hold for $j = i$ as well. 
\end{proof}



\begin{de}[Attaching a character to a torus element]\label{defn: character for ta}
For any $\ta\in\fa_\ell$ denote by $\chi_\ta$ the character on $\fa_\ell$ which is defined by \[
\mathtt{b}\mapsto\frac{\langle\ta,\mathtt{b}\rangle}{\langle\ta,\ta\rangle}. \]
Similarly, for any character $\chi\in\fa_\ell^*$ we denote by $\ta_\chi$ the element in $\fa_\ell$ which satisfies $\chi=\chi_{\ta_\chi}$. 

Since $\{\chi_1,\dots,\chi_r\}$ spans $\fa_\ell^*$, there exist $a_1,\dots, a_r\in\RR$ such that 
\begin{equation}\label{eq: chi_tau decomposition}
    \chi_{\ta}=\sum_{i=1}^r a_i\chi_i.
\end{equation}
Moreover, fixing $\Delta_\ell$ determines a positive $\ell$-Weyl chamber by \[
\fa_\ell^+:=\{\ta:\forall\alpha\in\Delta_\ell,\quad\alpha(\ta)>0\}. \]
For any $\ta$ there exists a choice of $\Delta_\ell$ so that  $\ta\in\fa_\ell^+$. In that case, the coefficients $a_1,\dots,a_r$ in \eqref{eq: chi_tau decomposition} are non-negative. 
\end{de}

The next property of $\fa_{\ell}^+$ follows from Definition \ref{def: fundumental weights}, \eqref{eq: weyl defn}, as the set of weights of a given representation is preserved under the action of the Weyl group. 
\begin{lem}\label{lem: highest weight pos Weyl}
If $\lambda\in\fa_{\ell}^*$ is the $\ell$-highest weight of some irreducible representation of $G$, then $\lambda\in\overline{\fa_{\ell}^+}$. 
\end{lem}

\begin{claim}\label{claim: parabolic expansion}
    For every $\ta \in \overline{\fa^+_\ell}$, one can write $\ta = \sum_{i=1}^r c_i \ta_{\chi_i}$, with $c_i \ge 0$.
    Then $P_{\ta}= \bigcap_{c_i>0} P_i$.
    In particular, $P_i = P_{\ta_{\chi_i}}$.
\end{claim}
\begin{proof}
    We prove the claim by showing the equality of the Lie algebras. 

    First, we show $\bigcap_{c_i>0} \operatorname{Lie}(P_i)\subseteq \operatorname{Lie}(P_{\ta})$. Let \[
    I_\ta=\{\lambda\in\Phi_{\ell}:\forall c_i>0,\lambda\not\le-\alpha_i\}.\]
    Then, by \eqref{eq: fundamental weight}, for any $\lambda\in I_{\ta}$ we have $\langle\lambda,\chi_i\rangle\ge0$, which implies 
    \begin{equation}\label{eq: diagonal on parabolic}
        \lambda(\ta)=\sum_{i=1}^r c_i\langle\lambda,\chi_i\rangle\ge0
    \end{equation} 
    Let $X\in\bigcap_{c_i>0}\operatorname{Lie}P_i$. Since $\bigcap_{c_i>0}\operatorname{Lie}P_i=\bigoplus_{\lambda\in I_\ta}\fg_\lambda$, we may write $X=\sum_{\lambda\in I_\ta}X_\lambda$, where for any $\lambda$ we have $X_\lambda\in\fg_\lambda$.  Then, we can compute
    \begin{align}\label{eq: parabolic element decom}
        \operatorname{Ad}(\exp(-t\ta))(X)&=\sum_{\lambda\in I_\ta}\exp(-t\lambda(\ta))X_\lambda.
    \end{align}
    By \eqref{eq: diagonal on parabolic}, for any $\lambda\in I_\ta$, the value of $\lim\exp(-t\lambda(\ta))$ is either $0$ or $1$, and so the above converges, implying $X\in\operatorname{Lie}(P_{\ta})$.  

    To see that $\bigcap_{c_i>0} \operatorname{Lie}(P_i)\supseteq \operatorname{Lie}(P_{\ta})$, we note that by \eqref{eq: fundamental weight} for any $\lambda\in\Phi_{\operatorname{Ad}}\setminus I_\ta$ we have $\langle\lambda,\chi_i\rangle<0$, which implies 
    \begin{equation*}
        \lambda(\ta)=\sum_{i=1}^r c_i\langle\lambda,\chi_i\rangle<0. 
    \end{equation*} 
    Hence, for a point $X$ with a non-trivial factor in $\bigoplus_{\lambda\in \Phi_{\operatorname{Ad}}\setminus I_\ta}\fg_\lambda$, the sum in \eqref{eq: parabolic element decom} does not converge.  
\end{proof}


\section{Busemann functions}\label{sec:Busemann functions} 
In this section we use the notation and results of \S \ref{sec:symmetric spaces} and \S\ref{sec:parabolic subgroups}. We study a semisimple Lie group $G$ of noncompact type and the corresponding symmetric space $M$.

The main goal of this section is to show equivalent descriptions for Busemann
functions. In particular, the second part of Theorem \ref{thm: main} follows from the last result of this section, Theorem \ref{thm: f_v descends to Busemann}. 

\begin{de}[Busemann functions]\label{defn: Busemann function}
Let $d$ be a right-invariant Riemannian metric on $M$.  
Given a geodesic ray $\gamma:[0,\infty)\rightarrow M$, the function $\beta_\gamma:M\rightarrow\RR$ defined by
\begin{equation}\label{eq: busemann def}
    \beta_\gamma(x)=\lim_{t\rightarrow\infty}(d(x,\gamma(t))-t)
\end{equation}
is called the \emph{Busemann function associated to $\gamma$} (see \cite[\S II.8.17]{Metric Spaces}). 
\end{de}

\begin{ex}[Busemann functions on Euclidean spaces]\label{ex: bus on euclid}
    Let $u\in \RR^\ell$ be a vector with $\|u\|=1$. Then, the Busemann function of the geodesic $t\mapsto tu$ is $v\mapsto -\bra u,v\ket$.
\end{ex}

\begin{de}[The modular function]\label{de: modular function}
Given a field $\ell$ and an $\ell$-algebraic group $P$, denote by $\delta_P:P\to \ell^\times$ \emph{the $\ell$-modular function associated with $P$}. 
That is, assuming $\mu_P$ is the left Haar measure on $P$, for a Borel subset $S\subseteq P$ with positive $\mu_P$-measure and $h\in P$, we have
\[\delta_P(h)=\frac{\mu_P(Sh^{-1})}{\mu_P(S)}\]
(see\cite[\S VIII.2]{knapp}).
Recall that $\delta_P$ is a group homomorphism.
\end{de}
The following known result gives a simple description of the modular function associated with $P$, if $P$ is a Lie group. 
\begin{lem}[{\cite[Prop. 8.27]{knapp}}]\label{lem: moduler is det}
If $P$ is a Lie group, then the modular function $\delta_{P}$ of $P$ is given by\[
\delta_P(h)=|\det(\Ad_P(h))|,\]
where $\Ad_{P}(h)$ is the adjoint action of $h$ on $\operatorname{Lie}(P)$.
\end{lem}

\begin{de}[Positive homomorphism]
\label{de: positive hom}
Let $P\subset G$ be a parabolic subgroup. A homomorphism $f:P\to \RR$ is called \emph{positive} if it is a linear combination with nonnegative coefficients of $\log \delta_{P'}$ for maximal parabolic subgroups $P'$ containing $P$.

This definition has an $\ell$-algebraic version for subfields $\ell\subseteq \RR$. If $G$ is an $\ell$-algebraic semisimple group and $P$ an $\ell$-parabolic subgroup in it, then a homomorphism $f:P(\RR)\to \RR$ is called \emph{$\ell$-positive} if it is a linear combination with nonnegative and rational coefficients of $\log \delta_{P'}$ for $\ell$-maximal parabolic subgroups $P'$ containing $P$. 
\end{de}

\begin{ex}[Example \ref{sssec: SLnR} continued]
    For $1\le m\le n-1$ let us denote by $P_{m}$ the parabolic subgroup of $\SL_n(\RR)$ of block upper triangular matrices with block sizes $m$ and $n-m$. That is \[
    P_m=\left\{\begin{pmatrix}
        A & B \\ 0 & C
    \end{pmatrix}:A\in M_{m,m},\: B\in M_{m,n-m},\: C\in M_{n-m,n-m}\right\},\]  
    where $M_{k,l}$ denotes the set of $k\times l$ matrices with entries in $\RR$.  
    Then, the modular function $\delta_{P_{m}}$ satisfies 
    \[\delta_{P_m}\left(\begin{pmatrix}
        A & B \\ 0 & C
    \end{pmatrix}\right) = \frac{\det A^{n-m}}{\det C^{m}} = \det A^{n}.\]
    Hence, if $n=n_1+n_2+\cdots +n_s$ and $P$ is the parabolic subgroup of all upper diagonal block matrices with respect to the $(n_i)_{i=1}^s$ block structure, then the positive homomorphisms consist of all maps $\sigma:P\to \RR$ such that for some $c_1\ge c_2\ge \cdots \ge c_s$, 
    \[\sigma(A) = \sum_{i=1}^s c_i \log |\det A_i|,\]
    where $A_i$ are the blocks of $A$ on the diagonal. 
\end{ex}

The main goal of this section is to prove the following theorem.


\begin{thm}\label{thm: class equiv busemann} 
The following three classes of functions are equivalent:
\begin{enumerate}[label=\emph{(\arabic*)}, ref=(\arabic*)]
    \item \textbf{Busemann functions:} The class of constant shifts of nonnegative multiples of 
    Busemann functions. 
    \item \textbf{Homomorphisms from parabolics:} \label{item: parabolic homomorphism} The class of functions which satisfies that $\pi(p)\mapsto \chi(p)$ for all $p\in P$, where $P$ is a parabolic subgroup of $G$, $\chi:P\to \RR$ is a positive homomorphims.
    \item \textbf{Lengths of highest weight vectors in fundumental representations:}\label{class: sum of funcs of vectors}  The class of functions of the form
    \begin{align*}
    \pi(g)\mapsto\sum_{i=1}^r c_i \log \|\varrho_i(gg_0)v_i\| + C,
    \end{align*}
    where $\varrho_i$ are $\RR$-fundamental representations, $g_0\in G$, and $v_i$ are the appropriate highest weight vectors, which correspond to maximal $\RR$-parabolic groups, and are defined in \S\ref{sec:fundamental representation}, for some choice of an  $\RR$-simple system, the $c_i$ are non-negative, and $C\in\RR$.
\end{enumerate}
\end{thm}

For the proof of Theorem \ref{thm: class equiv busemann} we need a good description of the modular functions as well as of the kernel of the Busemann functions. 


\begin{lem}
\label{lem: modular paravec behaviour} Let $1\le i\le r$. Then, for any $h\in P_{i}$,  $$\varrho_{i}(h) v_{i} = \delta_{P_{i}}(h)v_{i}.$$
\end{lem}
\begin{proof}
First, let $v_i$ be as in \S\ref{sec:fundamental representation}, a highest weight vector. Thus, the line spanned by $v_i$ is the highest weight eigenspace of $\varrho_i$, and is stabilized by $P_{i}$. That is, for any $h\in P_{i}$, the point $\varrho_{i}(h)v_{i}$ is a constant multiple of $v_{i}$. 

Second, using Lemma \ref{lem: moduler is det} and a standard decomposition of parabolic groups (see \cite[\S V.7]{knapp}), one can deduce that $|\det(\Ad_{P_i}(h))|$ is equal to the determinant of the adjoint action of $P_i$ on $\fn_i:=\bigoplus_{\langle\lambda,\chi_i\rangle>0}\fg_\lambda$. 
As the vector $v_{i}$ (as in \S\ref{sec:fundamental representation}) is defined to be the wedge product of a basis of $\fn_i$, the claim follows.   
\end{proof}

\begin{proof}[Proof of Theorem \ref{thm: class equiv busemann}]
    First note that the equivalence between classes (2) and (3) follows directly from Lemma \ref{lem: modular paravec behaviour}, and so it is enough to prove the equivalence between classes (1) and (3).
    
    Let $\beta_\gamma$ be a Busemann function, i.e., there exist $\ta\in \fa$ and $g\in K$ such that $\gamma$ is defined by $t\mapsto \pi(\exp(t\ta)g)$. 
    Since both classes (1) and (3) are invariant under the $G$ action, we may apply $\Ad(g)^{-1}$ to replace the geodesic with 
    $t\mapsto \pi(\exp(t\ta))$, and assume that $g = I$.
    Moreover, by Lemma \ref{lem:Weyl so dominant} and Remark \ref{rem:Weyl in K}, for some element $k\in K$ we may apply $\Ad(k)^{-1}$ to assume that $\ta\in-\overline{\fa^+}$.
    Hence, one can represent $-\ta = \sum_{i=1}^r c_i\ta_{\chi_i}$, with $\ta_{\chi_i}$ as in Definition \ref{defn: character for ta}.
    
    Let $C:=-\sum_{i=1}^r c_i\log \|v_i\|$, so that the desired equality 
    \begin{align}\label{eq: busman is parabolic}
        \beta_\gamma(\pi(g)) = \sum_{i=1}^r c_i \log \|\varrho_i(g)v_i\| + C,
    \end{align}
    holds for $g = I$.
    Since $\pi(A)$ is a flat, Example \ref{ex: bus on euclid} implies that \eqref{eq: busman is parabolic} holds for all $g\in A$.


    Let $P = P_{-\ta} \stackrel{\ref{claim: parabolic expansion}}{=} \bigcap_{c_i>0} P_i$, 
    and $U = U_{-\ta}$. 
    Then, for every $x\in M$ and $u\in U$ we have     \[\beta_\gamma(xu) = \lim_{t\rightarrow\infty}(d_M(xu,\gamma(t))-t) = \lim_{t\rightarrow\infty}(d_M(x,\gamma(t)u^{-1})-t).\]
    Since 
    \[\lim_{t\to \infty} d_M(\gamma(t)u^{-1}, \gamma(t)) = d_M(\pi(\exp(t\ta)u^{-1}\exp(-t\ta)), \pi(I)) \xrightarrow{t\to \infty  }0,\]
    we deduce that $\beta_\gamma$ is $U$ invariant. 
    Similarly, by Claim \ref{claim: U fixes highest weight}, the right $U$-action preserves the right-hand side of Eq. \eqref{eq: busman is parabolic}. 

    Let $k\in \stab_K\gamma$.
    Since $k$ preserves $\gamma$, we deduce that $k$ preserves $\beta_\gamma$. Since $\stab_K\gamma \subseteq P$ we deduce that 
    $k\in P$ and hence the right $k$ action preserves the right-hand side of \eqref{eq: busman is parabolic}.

    Since the norms on $V_i$ are $K$ invariant, Eq. \eqref{eq: busman is parabolic} depends only on $\pi(g)\in M$, and not on $g$. 


    Let $g\in G$ and represent $g = kup$, where $k\in K$, $u\in U_{-\ta}$ and $p\in T_{-\ta}$. 
    Since $p\in T_{-\ta}$, we deduce that $p = \exp(\ta')$ and $\ta'\in \fp$ commutes with $\ta$. Let $\fa' \subseteq \fp$ be a maximal abelian subalgebra containing $\ta, \ta'$. 
    By \ref{cor: transitive on Cartan}, there is an element $k\in \stab_{K}(\ta)$ such that $Ad_k(\fa) = \fa'$. 
    Since \eqref{eq: busman is parabolic} holds on $A = \exp(\ta)$, and is preserved by $\stab_{K}(\ta)$, we deduce that \eqref{eq: busman is parabolic} holds on $p$. 
    Since \eqref{eq: busman is parabolic} is preserved by multiplication from the left by $K$, and from the right by $p^{-1}up \in U$, we deduce that it holds on $g = kp p^{-1}up$. 
    The other direction, if a function is of class (3) then it is also of class (1) follows from the same computation. 
\end{proof}

\begin{remark}\label{rmk: busmann has parabolic}
  For a Busemann function $\beta_\gamma$ of a geodesic ray $\gamma:t\mapsto \pi(\exp(t\ta))$, for some $\ta\in \fp$, the parabolic subgroup in Class \ref{item: parabolic homomorphism} is $P_{-\ta}$. 
\end{remark}

\begin{thm}\label{thm: f_v descends to Busemann}
Let $\varrho:G\to \SL(V)$ be an $\RR$-representation of $G$, $v\in V\setminus \{0\}$ be a parabolic equivariant vector, and $P$ be the parabolic which stabilizes $\RR v$. Then, the homomorphism $f:P\to \RR$, $p\mapsto \log|\varrho(p)v/v|$ is positive. Here $\varrho(p)v/v$ is the unique scalar $\alpha \in \RR$ such that $\alpha v = \varrho(p)v$.
  Consequently, for every $K$ invariant norm $\|\cdot\|$ on $V$ the map $g\mapsto \log \|\varrho(g)v\|$ descends to a Busemann function on $M=K\backslash G$.
\end{thm}
\begin{proof}
Fix a maximal $\RR$-split torus $A \subset P$.
By Lemma \ref{lem:stabilized by parab}, since $P$ stabilizes $\RR v$,  $v$ is an $\RR$-highest weight vector for $\varrho$ with respect to some choice of a simple system. 
Without loss of generality, we may assume that, in the notation of \S\ref{sec:fundamental representation}, the minimal $\RR$-parabolic subgroup $P_0=\bigcap_{i=1}^r P_i$ is contained in $P$ (as otherwise it can be replaces with its conjugated to the one chosen here). 
Denote this $\RR$-highest weight by $\chi$. 
 Without loss of generality, we may assume $A$ is the one chosen in \S\ref{sec:fundamental representation} (as it is conjugated to the one chosen here).

As defined in \S\ref{sec:representations}, $\chi$ is a character on $\fa$. Equivalently, $\chi$ can be viewed as a character on $A$. Recall that for any $a\in A$ we have $\varrho(a) v = e^{\chi(a)}v$. Since $\varrho(p_1p_2) v=\varrho(p_1)\varrho(p_2) v$ and $\varrho(P)v=\RR v$, we can expend $\chi$ to a character on $P$ which satisfies\[
\varrho(p) v = e^{\chi(p)}v. \] 
By \eqref{eq: chi_tau decomposition} and Lemma \ref{lem: highest weight pos Weyl}, we have $\chi=\sum_{i=1}^r c_i\chi_i$ for some choice of non-negative $c_1,\dots,c_d$. That is,\[
\varrho(p) v = e^{\chi(p)}v= e^{\sum_{i=1}^r c_i\chi_i(p)}v= \prod_{i=1}^r \delta_{P_i}^{c_i}(p)v, \]
where the last equality follows from Lemma \ref{lem: modular paravec behaviour}. 
\end{proof}

\begin{ex}[Example \ref{sssec: SLnR} continued]
    Assuming $G = \SL_n(\RR)$, for any $1\le i\le n-1$ the fundamental weight $\chi_i$ is a scalar multiple of the highest weight of the exterior product representation (on $\bigwedge_i\RR^n$). 
    In particular, one may use theis representation instead of $\varrho_i$, and obtain a simplification of Class \ref{class: sum of funcs of vectors} in Theorem \ref{thm: class equiv busemann}:
    Let $u_1,...,u_n$ be a basis of $\RR^n$, $c_1,..., c_{n-1} \ge 0$ and $C>0$. Consider the class of functions of the form 
    \begin{align*}
        \pi(g)\mapsto \sum_{i=1}^{n-1} c_i \log \| (gu_1)\wedge (gu_2)\wedge\cdots \wedge (gu_i)\| + C.
    \end{align*}
\end{ex}

\section{The fastest shrinking geodesic} \label{sec:The fastest shrinking geodesic}
To construct the Busemann function in Theorem \ref{thm: main}\ref{part: main ge busemann}), we use the `fastest shrinking geodesic' of our function of interest. In this section, the fastest shrinking geodesic is constructed for a class of functions in the more general setting of a $\operatorname{CAT}(0)$-space, and the properties of such geodesic are studied in a special case, which is relevant to us, and is used in the next chapter. 

Recall the definition (and relevant notations) of a Hadamard space from \S\ref{sec: cat0 spaces}.
Fix a locally compact, Hadamard space $(M,d)$ and a point $o\in M$. 

For $s>0$, $x\in M$ denote by $B(x,s)$ the closed ball of radius $s$ around the point $x$.

\begin{lem}\label{lem: x_s exists and unique}
Let $f:M\rightarrow\RR$ be a convex function which is unbounded from below. Then, 
for any $x\in M$, $s>0$, the function $f$ attains a minimum on the closed ball $B(x,s)$ at a unique point on its boundary $\{y\in M:d(y,x) = s\}$. 
\end{lem}

\begin{proof} 
First, assume by contradiction that $f$ attains a minimum on $B(x,s)$ at a point $y_1$ with $d(e, y_1) < s$. In particular, there is a neighbourhood of $y_1$ which is contained in $B(x,s)$. Hence, $y_1$ is a local minimum of $f$. Since $f$ is convex, this implies that $y_1$ is a global minimum. A contradiction to the assumption that $f$ is unbounded from below. 

Second, assume by contradiction that $f$ attains its minimum on $B(x,s)$ at two points $y_1,y_2$. Since $M$ is a $CAT(0)$ space, it follows that the midpoint $y_3$ of $y_1$ and $y_2$ satisfies $d(y_3, x) < s$. 
Since $f$ is convex it follows that $f(y_3) \le f(y_2) = f(y_1)$. This implies that $y_3$ is another points in which $f$ attains a minimum, with $d(y_3, x) < s$. A contradiction to the previous discussion.
\end{proof}

Fix a convex, unbounded from below, function $f:M\to \RR$.

\begin{de}[Minimizing points of $f$ in balls]\label{de: minimizing points}
For any $s>0$ denote by $x_s$ the point which minimizes $f$ on $B(o,s)$. By Lemma \ref{lem: x_s exists and unique} the point $x_s$ is uniquely defined and satisfy $d(o, x_s) = s$. In particular, this definition implies $x_0=o$.
Denote by $\gamma_s:[0,s]\to M$ the geodesic connecting $o$ and $x_s$.
\end{de}

\begin{lem}\label{lem:a exists}
    The sequence $\frac{f(x_s)}{s}$ has a limit. 
\end{lem}

\begin{proof}
Let us note that by the definition of $x_s$ and the convexity of $f$, for any $s<t$ we have 
\begin{align*}
    \frac{f(x_{s})}{s}&\le \frac{f(\gamma_{t}(s))}{s}\le \frac{t-s}{t}\frac{f(\gamma_{t}(0))}{s}+ \frac{s}{t}\frac{f(\gamma_{t}(t))}{s}\\
    &= \left(\frac{1}{s}-\frac{1}{t}\right)f(o)+ \frac{f(x_t)}{t}. 
\end{align*}
That is, the shift $g(x)=f(x)-f(o)$ is a convex function, so that $\frac{g(x_s)}{s}$ is negative and non-decreasing. Hence, the sequence $\frac{g(x_s)}{s}$ has a limit, which implies that the sequence $\frac{f(x_s)}{s}$ also has a limit. 
\end{proof}

\begin{de}\label{de: decrease linearly}
We say that $f$ is \emph{decreasing linearly} if \[\lim_{s\to\infty} \frac{f(x_s)}{s}<0.\]
\end{de}

\begin{prop}\label{prop: one gamma}
If $f$ is decreasing linearly, then the sequence of geodesics $(\gamma_s)_{s>0}$ constructed in Definition \ref{de: minimizing points} converges to a geodesic $\gamma_\infty:[0,\infty)\to M$ pointwise. 
Moreover, setting $a:= - \lim_{s\to\infty} \frac{f(x_s)}{s}>0$, the geodesic $\gamma_\infty$ is the unique geodesic which satisfies 
\[ \gamma_\infty(0) = o\quad \text{and}\quad f(\gamma_\infty(s)) = -as\cdot(1+o(1)).\] 
\end{prop}
\begin{proof}
Let \begin{align}\label{eq: def a}
    a := -\lim_{s\to \infty} \frac{f(x_s)}{s}.
\end{align}
Since $f$ is decreasing linearly, we have $a>0$. 
Since the space of geodesic rays from $o$ is compact, the sequence $\{\gamma_s\}$ has a partial limit. Let $\gamma_\infty:[0,\infty)\to M$ be such a partial limit. 

Since $f$ is convex, for any $s>0$ we have 
\begin{align*}
f(\gamma_\infty(s)) &= \lim_{i\to \infty} f(\gamma_{s_i}(s))\\
&\le 
\lim_{i\to\infty}\left(\frac{s_i-s}{s_i} f(x_0) + \frac{s}{s_i} f(x_{s_i})\right)\\&= 
f(x_0) - s\cdot a,
\end{align*}
where the limit in the second line is taken over large enough $i$ so that $s_i>s$. 
Since $\gamma_\infty$ starts at $o$, we have that for any $s>0$,  $d(\gamma_\infty(s),o)=s$. 
Hence, by Definition \ref{de: minimizing points} for any $s>0$ we have 
\begin{equation}\label{eq: bound f x_s}
    f(\gamma_\infty(s))\ge f(x_s). 
\end{equation}
It now follows from \eqref{eq: def a} that for all $s>0$ 
$$f(\gamma_{\infty}(s)) = -as\cdot(1+o(1)).$$
Thus, $f\circ\gamma_\infty$ is of the claimed form. 

Let $\gamma_\infty':[0, \infty)\to M$ be another geodesic ray with $\gamma_\infty'(0) = o$ and 
$f(\gamma_\infty'(s)) = -as\cdot(1+o(1))$ for all $s\ge 0$.
Denote by $r$ the distance at time one between the two geodesics, i.e., $r := d_M(\gamma_\infty(1), \gamma_\infty'(1))$.
By Lemma \ref{lem: cat0 implication}(2), $d_M(\gamma_\infty(t), \gamma_\infty'(t)) \ge rt$ for every $t\ge 1$. 
Denote by $m_t$ the middle point between $\gamma_\infty(t)$ and $\gamma_\infty'(t)$. 
Then, by Lemma \ref{lem: cat0 implication}(3), we have
\begin{equation}\label{eq: upper bound gamma_infty}
	d_M(x_0, m_t) \le t \sqrt{1-r^2/4}. 
\end{equation}
On the other hand, since $f$ is convex it follows that $f(m_t) \le at(1 + o(1))$, and hence
\begin{equation*}
	f(x_{t\sqrt{1-r^2/4}}) \le -at(1 + o(1)).
\end{equation*}
Unless $r=0$, this together with \eqref{eq: upper bound gamma_infty} contradicts the definition of $a$. Hence $r=0$ and $\gamma_\infty = \gamma_\infty'$.
\end{proof}

In view of Proposition \ref{prop: one gamma} we define:
\begin{de}[Fastest shrinking geodesic]\label{de: fastest shrinking}
Let $(M, o , f)$ be a tuple of a locally compact Hadamard space $M$, a point $o$ in $M$, and a convex, linearly decreasing function $f:M\to \RR$. Define the \emph{fastest shrinking geodesic of $(M, o , f)$} to be $\gamma = \gamma_{M,o,f}:[0,\infty)\to M$, the limit of $\gamma_s$ defined as in Definition \ref{de: minimizing points}.
Then, for $a_{M,o,f}:= -\lim_{s\to \infty}\frac{1}{s}\min_{x\in B(s,o)}f(x)$ we have $f(\gamma(s)) = -a_{M,o,f}s\cdot (1+o(1))$ for all $s>0$. 
The constant $a_{M,o,f}$ is called the \emph{shrinking rate of f}.
\end{de}

To analyze symmetric spaces (see \S\ref{sec:symmetric spaces}) it is useful to consider the maximal flats in them. 
Hence, we show that restriction to `nice' subspaces does not change the fastest shrinking geodesic or the shrinking rate.

\begin{lem}[Restriction]
\label{lem: fsg restriction}
Let $(M, o , f)$ be a triplet as in Definition \ref{de: fastest shrinking}. 
If $Y\subseteq M$ is a sub Hadamard space which contains the image of $\gamma_{M, o, f}$ then $\gamma_{M, o, f} = \gamma_{Y, o, f|_{Y}}$ and $a_{M, o, f} = a_{Y, o, f|_{Y}}$. 
\end{lem}

\begin{proof}
We denote $\gamma:=\gamma_{M, o, f}$. 
By definition, $$a_{M, o, f} \le a_{Y, o, f|_{Y}} \le a_{\Im\gamma, o, f\circ \gamma}.$$
Proposition \ref{prop: one gamma} implies that $a_{\Im\gamma, o, f\circ \gamma} = a_{M, o, f}$. Hence, we deduce $a_{M, o, f} = a_{Y, o, f|_{Y}}$. 
Then, the uniqueness of the limit of geodesics in Proposition \ref{prop: one gamma} implies that $\gamma_{Y, o, f|_{Y}}=\gamma$.
\end{proof}

The next result follows directly from Proposition \ref{prop: one gamma}. 
\begin{lem}[Bounded shifts]\label{lem: fsg O(1)}
Let $(M, o, f)$ be a triplet as in Definition \ref{de: fastest shrinking}, and $\tilde{f}:M\to \RR$ be a convex function such that $|f - \tilde{f}|$ is bounded.  Then $(M, o, f)$ satisfies the assumptions of Definition \ref{de: fastest shrinking} and  \[
\gamma_{M, o, f} = \gamma_{M, o, \tilde{f}},\quad a_{M, o, f} = a_{M, o, \tilde{f}}.\]
\end{lem}

\subsection{Piecewise linear convex functions on Euclidean spaces}\label{sec:piecewise linear convex functions}

Our goal in defining the fastest shrinking geodesic is to use them to study, and more specifically bound the `shrink-rate functions', which are defined in the next section. 
The restriction of these functions to maximal flats (see Definition \ref{de: flats}) are piecewise linear, up to bounded error. Moreover, the metric on the maximal flats is Euclidean.  
Therefore, we dedicated this subsection to study the fastest shrinking geodesics in this setting.

In this subsection, we assume that $M$ is the Euclidean space $\RR^n$ with the standard Euclidean metric. 

\begin{claim}\label{claim: fsg of max of linears}
Let $V_0\subseteq \RR^n$ and $(a_v)_{v\in V_0}\subseteq \RR^n$ be a collection of real numbers. Define the function \[f(x):= \max_{v\in V_0}(\bra x, v\ket + a_v).\] 
Then, $f$ is unbounded from below if and only if 
\begin{equation}\label{eq: 0 not in convex}
    0\notin \conv (V_0).
\end{equation}
Moreover, assuming \eqref{eq: 0 not in convex}, and denoting by $u$ the closest point to $0$ in $\conv(V_0)$, we have: 
\begin{enumerate}
    \item $f$ decreases linearly,
    \item The fastest  shrinking geodesic for $f$ is \[\gamma_{\RR^n, 0, f}(t) = -tu / \|u\|,\] 
    \item $a_{M, o, f}= -\|u\|$, and
    \item There exists $C\in \RR$, such that 
    \begin{align}\label{eq: bound piecewise linear}
    f(x) \ge \bra x, u\ket + C.
\end{align}
\end{enumerate}
\end{claim}

\begin{proof}
If $0\in \conv(V_0)$, then for some non-negative $(c_v)_{v\in V_0}$ with $\sum_{v\in V_0}c_v = 1$ we have $\sum_{v\in V_0}^n c_v v = 0$. Then, \[
f(x)= \max_{v\in V_0}(\bra x, v\ket + a_v) \ge \sum_{v\in V_0}^n c_v(\bra x, v\ket + a_v) = \sum_{v\in V_0} c_va_v.\]
is a bound from below on $f$.

Now, assume $0\notin \conv(V_0)$. We show that $f$ decreases linearly, which also implies that $f$ is unbounded from below. By Definition \ref{de: minimizing points}, it is enough to find a direction in which $f$ decreasing linearly, i.e. a vector $v$, such that $f(sv)\le cs$ for some $c<0$ and all large enough $s$. 

Let $u$ be the closest point to $0$ in the convex hull $\conv(V_0)$. 
Since $u$ is of minimal length, for any $v\in V_0$ and $p\in(0,1)$ we have
\begin{align*}
    \norm{u}^2&\le\norm{pv+(1-p)u}^2\\
    &=p^2\norm{v}^2+2p(1-p)\langle u,v \rangle+(1-p)^2\norm{u}^2,
\end{align*}
which implies \[
\langle u,v \rangle\ge \norm{u}^2-\frac{p}{2(1-p)}(\norm{v}^2-\norm{u}^2). \]
Thus, we deduce $\langle u,v \rangle\ge \norm{u}^2$, for all $v\in V_0$. 
We may now compute \[f(-tu) = \max_{1\le i\le n}(-t\bra u, v\ket + a_v) \le -t\|u\|^2 + \max_{v\in V_0}a_v,\] 
which decreases linearly, proving (1).

According to Proposition \ref{prop: one gamma} there is a minimal shrinking geodesic and $a_{\RR^n, 0, f} \le -\|u\|$.
On the one hand,
for all $x\in \RR^n$ we have 
\[f(x)= \max_{v\in V_0}(\bra x, v\ket + a_v) \ge \sum_{v\in V_0} c_v(\bra x, v\ket + a_v) = \bra x, u\ket + \sum_{v\in V_0} c_va_v,\]
which implies (4) and 
$a_{\RR^n, 0, f} \ge -\|u\|$. 
Thus, (2) and (3) are also satisfied.
\end{proof}

The following lemma shows that the constant of the lower bound in \eqref{eq: bound piecewise linear} is (in a sense) continuous in the vectors which define $f$ and can be chosen to be uniform on compact sets.

\begin{lem}\label{lem: compact lower bound}
Let $V_0\subseteq \RR^n$ be a finite set, $K$ be a compact set, and $(r_v)_{v\in V_0}$ be a collection of continuous functions, $r_v:K\to \RR\cup \{-\infty\}$. 
Assume that for each $k\in K$ the function $$f_k:\RR^n\to \RR,\quad f_k(x) = \max_{v\in V_0}(\bra x, v\ket + r_v(k))$$ is well defined (i.e.  $f_k$ does not attain the value $-\infty$), and the point of minimal length for each $k$ \[u\in \conv(\{v\in V_0:r_v(k)>-\infty\})\]
is independent of $k$, i.e., a constant vector. 
Then, there exists $C\in \RR$ such that for all $k\in K, x\in \RR^n$ we have 
\begin{align}\label{eq: bound f by linear}
f_k(x) \ge \bra x, u\ket +C.
\end{align}
\end{lem}

\begin{proof}
By the definition of $u$, the value 
\[\xi(k):=\max_{\substack{V_1\subseteq V_0\\u\in \conv (V_1)}}\min_{v\in V_1} r_v(k),\]
is not $-\infty$ for every $k\in K$.
Since $\xi$ defines a continuous function $\xi:K\to \RR$, it attains a minimum, which we denote by $C$.

We now show that for every $k\in K,\, x\in \RR^n$, Equation \eqref{eq: bound f by linear} is satisfied. 
Indeed, for some $V_1\subseteq V_0$ with $u\in \conv (V_1)$ we have $\min_{v\in V_1}r_v(k) \ge C$.
Since $u\in \conv (V_1)$, there exists a convex combination $\sum_{v\in V_1}c_v v = u$. Then, 
\begin{align*}
    f_k(x) &= 
    \max_{v\in V_0}(\bra x, v\ket + r_v(k)) \ge 
    \sum_{v\in V_1} c_v(\bra x, v\ket + r_v(k))
    \\
    &= \bra x, u\ket + \sum_{v\in V_1}c_v r_v(k) \ge \bra x, u\ket  + C,
\end{align*}
as wanted. 
\end{proof}

\begin{remark}
Lemma \ref{lem: compact lower bound} does not hold if one allows $V_0$ to depend on $K$.
For example, by taking $K = [0,1]$, and for each $k\in K$, \begin{align*}
    &V_k=\begin{cases}
\{(1,-k), (1,1)\}&\text{if }k>0,\\
\{(1,0)\}&\text{if }k=0,
\end{cases}\\
&r_{0,(1,0)}=0,\quad r_{k,(1,-k)}=0,\quad r_{k,(1,1)}=\begin{cases}
\frac{1}{k^2} &\text{if }k>0,\\
0&\text{if }k=0,\end{cases}
\end{align*}
we get that for the sequence of functions
\[f_k:\RR^2\to \RR, \quad f_k(x,y) = x + \begin{cases}
\max\left(-ky, y- \frac{1}{k^2}\right)& \text{if } k > 0,\\
0 & \text{if } k = 0,
\end{cases}\]
the result of Lemma \ref{lem: compact lower bound} fails.
\end{remark}

The following Claim will be used in \S\ref{sec: relation to kempf}.
\begin{claim}\label{claim: lim continuous}
  Let $f:\RR^n\to \RR$ be a convex and piecewise linear function with finitely many slopes. Then the map 
  $v\mapsto \lim_{t\to \infty}\frac{f(tv)}{t}$ is continuous. 
\end{claim}
\begin{proof}
   Since $f$ is a piecewise linear function, there exist a finite set $V_0\subseteq \RR^n$ and a collection of real numbers $(a_v)_{v\in V_0}$ such that $f(x) = \max_{v\in V_0}(\bra x, v\ket + a_v)$ for every $v\in \RR^n$. Then, for any $x\in\RR^n$ 
   \begin{align*}
       \lim_{t\to \infty}\frac{f(tx)}{t} = \max_{v\in V_0}\bra x, v\ket,
   \end{align*}
   which is a continuous function. 
\end{proof}

\section{The shrink-rate function} \label{sec:shrink-rate function}

Let $\varrho:G\rightarrow \gl(V)$ be an $\RR$-representation, $v\in V$, and $\|\cdot\|$ be a norm on $V$ as in Lemma \ref{lem: construction of bilinear form}.  
The map $G\to \RR$ defined by $g\mapsto\log\|\varrho(g)v\|$ is invariant under the right action of $K$. Hence, it defines a map $f_v:M\to \RR$. The function $f_v$ is called \emph{the shrink-rate function of $v$}. 

Recall the representation theory notations from \S\ref{sec:representations}, the definition of the Busemann function from \S\ref{sec:Busemann functions}, and the definition of the fastest shrinking geodesic from \S\ref{sec:The fastest shrinking geodesic}. 

The main goal of this section is to analyse $f_v$ assuming $v$ is unstable. We prove Theorem \ref{thm: main} Part \ref{part: main ge busemann}, by proving the next theorem.  

\begin{thm}[Busemann bounds $f_v$ from below]\label{thm: busemann bound}
The tuple $(M,o,f_v)$ satisfies the assumptions of Definition \ref{de: fastest shrinking}. In particular, there exists a fastest shrinking geodesic $\gamma=\gamma_{M, o, f_v}:[0,\infty)\to M$ and shrinking rate $a=a_{M,o,f_v}>0$ for it. Moreover, there exists $C\in \RR$ such that for every $x\in M$, we have 
\begin{align}\label{ineq: f more then busemann}
f_v(x) \ge a\beta_\gamma(x) + C.
\end{align}
\end{thm}
\begin{remark}
  We distinguish our function $f_v$ with the pathological in \S\ref{ssec:further research} using the fact that our function $f_v$ decay rate is almost linear, that is, $f(\gamma_{M, \pi(e), f}(s)) = -as + O(1)$. 
\end{remark}
\begin{remark}
    We use norms on $V$ that satisfy Lemma \ref{lem: construction of bilinear form} to obtain the convexity of $f_v$. If we did not assume that the norm satisfies this condition and is only $K$-invariant, \eqref{ineq: f more then busemann} would still hold with a different constant $C$, as all norms on $V$ are equivalent. 
\end{remark}


\begin{obs}[Convexity of $f_v$]\label{obs: convexity of f_v}
The choice of quadratic form in Lemma \ref{lem: construction of bilinear form} implies that $f_v$ is convex. Moreover, for every maximal flat $\pi(Ag)$, $g\in K$, we have 
\[f_v(\pi(ag)) = \frac12\log \sum_{\lambda \in \Phi_\varrho} \lambda(a)^2\|(gv)_\lambda\|^2 = \max_{\lambda\in\Phi_\varrho}\left(\log \lambda(a)+r_\lambda\right)+O(1),\]
where $r_\lambda = \log \|(gv)_\lambda\| \in \RR\cup \{-\infty\}$ depends continuously on $g$ and $v$.
\end{obs}

Recall that $v\in V$ is called \emph{unstable} if $0\in\overline{Gv}\setminus Gv$, where $\overline{Gv}$ denotes the Zariski-closure of $Gv$. 
\begin{claim}\label{claim: f_v decreases linearly}
If $v$ is unstable, then the function $f_v$ decreases linearly (see Definition \ref{de: decrease linearly}).
\end{claim}
\begin{remark}
This claim follows from \cite{kempf}, but we prove it differently to illustrate the techniques in a simpler way.
\end{remark}

\begin{proof}
Recall that by Definition \ref{de: flats} the metric on $M$ (explicitly defined in \S\ref{sec: explicit construction}) restricted to any flat is Euclidean. 
Recall that we denoted by $\fa$ the Lie algebra of $A$.

First, note that every maximal flat of $M$ is of form $F_k:=\pi(Ak)$, for some $k\in K$. 
Moreover, the restriction of $f_v$ to $F_k$ can be \emph{approximated} by 
\begin{align}\label{eq: defn of f_k}
    \tilde{f}_{k}(\ta) = 
    \max_{\lambda\in\Phi_\varrho}\left(\bra \ta_\lambda ,\ta\ket+r_\lambda(k)\right),
\end{align}
where $\ta_\lambda$ is the unique vector such that $\bra \ta_\lambda,\ta\ket = \lambda(\ta)$ and $r_\lambda(k) = \log \|(kv)_\lambda\| \in \RR\cup \{-\infty\}$.
By approximated we mean that for $\ta\in\fa$ the difference 
$f_{v}(\exp(\ta)k)-\tilde{f}_{k}(\ta)$ is bounded. This approximation holds since $F_k$ is a maximal flat, by Observation \ref{obs: convexity of f_v}, the explicit construction of the metric in \S\ref{sec: explicit construction}, and the definition of a maximal flat (see Definition \ref{de: flats}). Note also that here $\lambda$ is viewed as an additive character on $\fa$. In particular, if $\tilde{\lambda}$ is the same character viewed as a multiplicative character on $A$ (i.e., as in Observation \ref{obs: convexity of f_v}), then $\lambda(\ta)=\log\tilde\lambda (\exp(\ta))$. 

Next, assume that $f_v$ is unbounded from below on $F_k$ for some $k\in K$. 
Then, in particular, $\tilde{f}_{k}$ is unbounded from below.
Claim \ref{claim: fsg of max of linears} implies that $\tilde{f}_k$ decreases linearly, and hence also $f_v$.

Last, we assume that $f_v$ is bounded from below on all maximal flats $F_k$, $k\in K$ and show that it implies that $v$ is stable, i.e., $0\notin\overline{Gv}\setminus Gv$. 
It follows from the assumption that $\tilde{f}_k$ are bounded from below as well. By Claim \ref{claim: fsg of max of linears} for all $k\in K$, \[0\in \conv(\{v_\lambda: \lambda \in\Phi_\varrho, r_\lambda(k) \ge 0\}.\]
Lemma \ref{lem: compact lower bound} implies that there is a uniform lower bound on $\tilde{f}_k$ for all $k\in K$, and implies that $f_v$ has a uniform lower bound on $\bigcup_{k\in K}F_k$.
Using the Cartan decomposition of $G$ (see Definition \ref{de: cartan decomposition}), we have $\bigcup_{k\in K}F_k = M$, we conclude that $v$ is stable. 
\end{proof}

Note that Observation \ref{obs: convexity of f_v} and Claim \ref{claim: f_v decreases linearly} prove the first part of Theorem \ref{thm: busemann bound}. That is, the tuple $(M,o,f_v)$ satisfies the assumptions of Definition \ref{de: fastest shrinking}.
Therefore, set $\gamma$ to be the fastest shrinking geodesic of $(M,o,f_v)$, and $a>0$ satisfy the conclusion of Proposition \ref{prop: one gamma} for $f_v$. 
In view of Observation \ref{obs: convexity of f_v}, the study of the restriction of shrink-rate functions to maximal flats is reduced to \S\ref{sec:piecewise linear convex functions}.

\subsection{Lower bound on maximal flat} 
\label{ssub:proof_of_theorem_thm: busemann bound}

The first step in the proof of Theorem \ref{thm: busemann bound} is showing that Equation \eqref{ineq: f more then busemann} is satisfied on maximal flats (see \S\ref{sec:symmetric spaces} for the definition of a maximal flat).

\begin{prop}\label{prop: bound f on flats}
There exists a constant $C>0$, which only depends on $v$, so that for every maximal flat $\gamma\subseteq F\subset M$ and for every $x\in F$, inequality \eqref{ineq: f more then busemann} holds.
\end{prop}
\begin{proof}
Let $F:=\pi(Ag)$, $g\in K$, be a maximal flat containing $\gamma$, and $K_0\subseteq K$ be the stabilizer of $\gamma$. By Claim \ref{claim: compact acts transitively on flats}, the group $K_0$ acts transitively on the set of maximal flats containing $\gamma$.

As in the proof of Claim \ref{claim: f_v decreases linearly}, since $F$ is a maximal flat, the metric constructed on it in \S\ref{sec: explicit construction} is Euclidean. By Corollary \ref{cor: geodesics description} we may assume that for some  $\ta\in\fa$, $\gamma(t)=\pi(\exp(\ta)g)$. 
Moreover, by Example \ref{ex: bus on euclid} the restriction of $\beta_\gamma$ to $F$ is defined by 
\begin{equation}\label{eq: busem for flats}
    \beta_\gamma(\pi(\exp(\mathtt{b})g)) = -\bra \ta,\mathtt{b}\ket.
\end{equation}

Fix $k_0\in K_0$. We wish to show that each $p\in F_{k_0}:=\pi(Agk_0)$ satisfies \[f_v(p)\ge a\beta_\gamma(p) + C.\]
Write $p = \exp(\mathtt{b})gk_0$ for $\mathtt{b}\in \fa$. Since $K_0$ preserves $\beta_\gamma$, by \eqref{eq: busem for flats}, this is equivalent to showing that \[
f_{\varrho(k_0)v}(p)\ge -\bra \ta, \mathtt{b} \ket + C. \]

Since $f_{\varrho(k_0)v}$ is a shift (by $k_0$) of the function $f_v$, by Lemma \ref{lem: fsg restriction}, the fastest shrinking geodesic of the function $f_{\varrho(k_0)v}\mid_{F_{k_0}}$ is $\gamma_{gk_0}(t)=\pi(\exp(\ta)gk_0)$ and the shrinking rate of it is $a$. 
By Observation \ref{obs: convexity of f_v} it follows that for any $p=\exp(\mathtt{b})gk_0\in F_{k_0}$ $f_{\varrho(k_0)v}(p)-\tilde{f}_{gk_0}(\mathtt{b})$ is bounded, where for $k:=gk_0$ the function $\tilde{f}_{k}$ is defined as in \eqref{eq: defn of f_k}. 
It follows from Lemma \ref{lem: fsg O(1)}  that the fastest shrinking geodesic of $\tilde{f}_k$ is $\gamma_k$ and the shrinking rate of it is $a$.
The result now follows from Lemma \ref{lem: compact lower bound}. 
\end{proof}

\subsection{Proof of the second part of Theorem \ref{thm: busemann bound}}\label{sec: proof of busemann bound}

The following proposition is a generalization of the fact that in a rank-$1$ space, for every geodesic ray $\gamma$ and a point $p$ not in the geodesic, there is another geodesic $\gamma'$ through $p$ such that $d(\gamma'(t), \gamma)\xrightarrow{t\to \infty}0$.

\begin{prop}\label{prop: geometry flats}
    For every geodesic ray $\gamma:[0,\infty)\to M$ and a point $x\in M$, there exists a maximal flat $\gamma\subseteq F\subseteq M$ and two geodesics $\gamma_1:\RR\to F$ and $\gamma_2:\RR\to M$ such that $\gamma_1$ is parallel to $\gamma$ in $F$, $\lim_{t\to \infty}d_M(\gamma_1(t), \gamma_2(t)) = 0$, and $\gamma_2(0) = x$.
\end{prop}
\begin{proof}
    Up to translation of the we may assume that $\gamma$ is defined by \[\gamma(t)=\pi(\exp(t\ta)),\quad\text{for}\quad t\in \RR,\]
    for some $\ta \in \fp$.
    Write $x=\pi(g_0)\in M$ for some $g_0\in G$.
    By Lemma \ref{lem: Iwasawa}, the element $g_0$ can be decomposed as $g_0=ktu$, $k\in K$, $t\in T_{-\ta}$, $u\in U_{-\ta}$. Since $U_{-\ta}$ is stabilized by $T_{-\ta}$, this is equivalent to $g_0 = k p_2 p_1$, where $k\in K$, $p_2\in U_{-\ta}$, $p_1\in T_{-\ta}$. 
    Let $\gamma_{1}$ be the geodesic defined by $t\mapsto \pi(\exp(t\ta)p_1)$. 
    Let $\gamma_{2}$ be the geodesic defined by $t\mapsto \pi(\exp(t\ta)p_2 p_1)$. 
    Since $p_1 = \exp(\tp_1)$ for some $\tp_1\in \fp$ that commutes with $\exp(\ta)$, 
    it follows that $\tp_1$ and $\ta$ lie in some maximal abelian $\fa'<\fp$.
    Then, $A' = \exp(\fa')$ is a Cartan torus such that $\pi(A')$ contains $\gamma$ and $\gamma_1$. 
    By Theorem \ref{thm: maximal flats}, $F = \pi(A')$ is maximal flat containing the geodesic $\gamma_1$. 
    It is left to prove the limit,
    \begin{align*}
        d_M(\gamma_1(t), \gamma_2(t)) &= 
        d_M(\pi(\exp(t\ta)), \pi(\exp(t\ta)p_2)) \\&= 
    d_M\left(\pi(I), \pi(\exp(t\ta)p_2\exp(-t\ta))\right) \xrightarrow{t\to \infty} 0.
    \end{align*}
    The second equality holds because the right $G$ action on $M$ is by isometries, and the limit is true since $p_2\in U_{-\ta}$
\end{proof}

Let $C\in \RR$ be the constant that satisfies the conclusion of Proposition \ref{prop: bound f on flats}. Fix $x \in M$.
We will show that $f_v(x_0) \ge C + a\beta_{\gamma}(x_0)$. 

Let $\gamma_1, \gamma_2, F$ be as in Proposition \ref{prop: geometry flats}. 
Let us study $\beta_\gamma$ on $\gamma_1$ and $\gamma_2$. 
Let $i=1,2$.
Since $\beta_\gamma$ is convex, we deduce that $t\mapsto \beta_\gamma(\gamma_i(t))$ is convex.
Since $\beta_\gamma$ is $1$-Lipschitz, we deduce that $t\mapsto \beta_\gamma(\gamma_i(t))$ is $1$-Lipschitz.
Since $\beta_\gamma$ is Lipschitz-continuous, and $d_M(\gamma_i(t), \gamma(t))$ id bounded for $t\ge 0$, we deduce that \[\beta_\gamma(\gamma_i(t)) + t = \beta_\gamma(\gamma_i(t)) - \beta_\gamma(\gamma(t))\] is bounded for $t\ge 0$.
\begin{obs}\label{obs: dec}
    If $f:\RR\to \RR$ is a convex, such that $\lim_{t\to \infty}\frac{f(t)}{t} = b$ exists. $f(t)-tb$ is decreasing. 
\end{obs}
Since $\beta_\gamma(\gamma_i(t))+t$ is bounded for $t\ge 0$, observation \ref{obs: dec} implies that $\beta_\gamma(\gamma_i(t))+t$ is decreasing. 
Since $\beta_\gamma(\gamma_i(t))$ is $1$-Lipschitz, we deduce that $\beta_\gamma(\gamma_i(t))+t$ is constant.
Since $d_M(\gamma_{1}(t), \gamma_2(t)) \xrightarrow{t\to \infty} 0$, and $\beta_\gamma(\gamma_i(t))$ is Lipschitz-continuous, we deduce that $\beta_\gamma(\gamma_1(t))+t = \beta_\gamma(\gamma_2(t))+t$, and denote this constant by $C'$

We turn to study $f_v$. 
By Claim \ref{prop: bound f on flats}, we have that \begin{align}\label{eq: fv gamma1}
    f_v(\gamma_1(t)) \ge a\beta_\gamma(\gamma_1(t)) + C = a(C'-t) + C. 
\end{align}
By Proposition \ref{prop: one gamma}, we have that $\lim_{t\to \infty}f_v(\gamma(t)) / t = -a$. 
For $i=1,2$, and $t\ge \infty$, since $d_M(\gamma_i(t), \gamma(t))$ is bounded and $f_v$ is Lipschitz-continuous, we deduce that $\lim_{t\to \infty}f_v(\gamma_i(t)) / t = -a$. 
Observation \ref{obs: dec} implies that  
\begin{align*}
    f_v(\gamma_i(t)) + at
\end{align*}
is decreasing. 
By \eqref{eq: fv gamma1}, we deduce that $f_v(\gamma_1(t)) + at \ge aC' + C$, and in particular, $\lim_{t\to \infty} f_v(\gamma_1(t)) + at$ exists, and is greater than $aC' + C$.  Since $d_M(\gamma_1(t), \gamma_2(t)) \xrightarrow{t\to \infty} 0$, we deduce that $\lim_{t\to \infty} f_v(\gamma_2(t)) + at = \lim_{t\to \infty} f_v(\gamma_1(t)) + at$ exists.
Since $f_v(\gamma_2(t)) + at$ is decreasing, we deduce that $f_v(\gamma_2(t)) + at \ge aC' + C$. 
This implies that 
\begin{align*}
    f_v(\gamma_2(t)) \ge a(C'-t) + C = a\beta_\gamma(\gamma_2(t)) + C,
\end{align*}
Substituting $t=0$, we obtain Eq. \eqref{ineq: f more then busemann} for $x$.

\qedhere\qed

\section{An algebraic interpretation and a result by Kempf}
  \label{sec: relation to kempf}
  Inspecting our setting from an algebraic point of view, one can use Kempf \cite{kempf} to obtain Theorem \ref{thm: busemann bound}, which is a more explicit version of Theorem \ref{thm: main} Part \ref{part: main ge busemann}. In this section, we discuss the algebraic analog of the geometric notions and results presented in previous sections. Specifically, we show Theorem \ref{thm: class equiv busemann alg}, which is an algebraic analog of Theorem \ref{thm: class equiv busemann}, and Theorem \ref{lem: kempf is fastest}, which is an algebraic amplification of the notion of the fastest shrinking geodesic. 
  
  \subsection{Algebraic analogous of geometric claims}
  Kempf studied group homomorphisms in $\Hom_\ell(\GG_m,G)$, for $\ell \subseteq \RR$, which are algebraic analogous of the geometric notion of geodesics. The following claim describes the connection between them:
  
  \begin{claim}[Connection between geodesics and group homomorphism]
    \label{claim: morphism is geodesic}
    Let $G$ an $\RR$-algebraic semisimple group (as in \S\ref{sec:symmetric spaces}). 
    For every nontrivial homomorphism $\tau:\GG_m\to G$ there is a unique $\ta \in \fp$ such that the homomorphism $t\mapsto \tau(\exp(t))$ is a conjugate, by an element in $P_{\ta}$ to the homomorphism $t\mapsto \exp(t\ta)$ (see \S\ref{sec: explicit construction} for the definition of $\fp$ and \S\ref{sec:parabolic subgroups} for the definition of $P_\ta$). 
  \end{claim}
  \begin{remark}
  A similar claim could be made for a general homomorphism in $\Hom(\RR, G)$, but since the adjoint action of the homomorphism may have non-real eigenvalues, the claim and its proof are more elaborate.
  \end{remark}
  \begin{remark}\label{rem: parabolic from a}
    In the setting of Claim \ref{claim: morphism is geodesic}, $P_\ta$ can be defined using $\tau$ by \[P_\ta=P_\tau := \{g\in G:\tau(t)g\tau(t^{-1}) \text{ converges as }t\to 0\}.\]
  \end{remark}
  \begin{proof}[Proof of Claim \ref{claim: morphism is geodesic}]
    Let $\tau:\GG_m\to G$ be a nontrivial homomorphism. 
    First note that by 
    \cite[Prop. 2.6]{mumford}, $P_\tau$, as defined in Remark \ref{rem: parabolic from a}, is parabolic. 
    
    Now, $\tau(\GG_m)$ is a torus of $G$. By
    \cite[Thm 15.13]{borel3}
    there exists $g\in G$ so that $g\tau(\GG_m)g^{-1}\subseteq A = \exp \fa$ for some Cartan torus $\fa<\fp$. 
    Hence there is $\ta'\in \fa$ such that $g\tau(\exp(s))g^{-1} = \exp(s\ta')$.
    By Lemma \ref{lem: Iwasawa} we obtain that $G= K P_\tau $.
    Decompose $g = kp$ with $k\in K$ and $p\in P_{\tau}$. 
    
    Since $\fp$ is invariant under $K$ we deduce that $\ta = \Ad_k(\ta')\in \fp$ satisfies 
    $p\tau(\exp(s))p^{-1} = \exp(s\ta)$ for all $s\in \RR$. Since $\tau$ is nontrivial, $\ta \neq 0$.
    
    Note that $\gamma: s\mapsto \pi(\exp(s\ta / \|\ta\|))$ is a geodesic. 
    The definition of $P_\tau$ implies that 
    \begin{align*}
        d(\gamma(s), \pi(\tau(\exp(s/\|\ta\|)))) 
    &= d(\pi(\exp(t\ta / \|\ta\|)), \pi(\tau(\exp(s/\|\ta\|))))
    \\&= d(\pi(p\tau(\exp(s/\|\ta\|))p^{-1}\tau(\exp(-s/\|\ta\|))), \pi(1))
    \end{align*}
    is bounded as $s\to -\infty$.
    Lemma \ref{lem: cat0 implication} implies that this defines $\gamma$ uniquely.
  \end{proof}
  \begin{de}
  Let $\tau\in \Hom_\RR(\GG_m,G)$ be a non-trivial homomorphism and $\ta$ be as in Claim \ref{claim: morphism is geodesic}. 
  The geodesic ray $\gamma_\tau$ defined by $\gamma_\tau(t) := \pi(\exp(t\ta/\|\ta\|))$ is called \emph{the geodesic ray associated to $\tau$} (where $\norm{\cdot}$ is as defined in \S\ref{sec: explicit construction}), and $\|\ta\|$ is called \emph{the renormalization constant of $\tau$}. Remark \ref{rem: norm is integral} below shows that $\|\ta\|^2$ is an integer.
  \end{de}
  \begin{remark}
    Note that not all geodesic rays are associated to a homomorphism.
    For example, take $G = \SL_3(\RR),\, K=SO(3)$, and \[\gamma(t) = \pi(\exp(t\diag(1, \sqrt{2}, -1-\sqrt{2}))),\] 
    then $\gamma$ does arise from any algebraic homomorphism in $\Hom_\RR(\GG_m, G)$.
  \end{remark}
  
  From here on we assume that $G$ is defined over a fixed field $\ell\subseteq \RR$. 
  \begin{de}[$\ell$-algebraic geodesics and Busemann functions]\label{de: k-algebraic}
    A geodesic ray $\gamma:[0,\infty)\to M$ is called \emph{$\ell$-algebraic} if it is associated to an $\ell$-algebraic group homomorphism $\tau\in \Hom_\ell(\GG_m, G)$. 
    In this case, let $\beta_\gamma$ be the Busemann function defined by $\gamma$ (as in \eqref{eq: busemann def}) and $a_\tau$ be the renormalization constant of $\tau$.
    For every $q\in \QQ^{>0}$ the function $qa_\tau\beta_\gamma$ is called the \emph{$\ell$-renormalized Busemann function}.  
  \end{de}
  We can now provide an algebraic analogous of Theorem \ref{thm: class equiv busemann}.
  
  \begin{thm}\label{thm: class equiv busemann alg} 
  The following three classes of functions are equal:
  \begin{enumerate}[label=\emph{(\arabic*)}, ref=(\arabic*)]
    \item \textbf{Busemann functions:} The class of constant shifts of
    $\ell$-renormalized Busemann functions. 
    \item \textbf{Homomorphisms from parabolics:} \label{cond: homomorphisms from parabolics} The class of functions which are projections of $\ell$-positive homomorphisms $P\to \RR$, where $P$ is an $\ell$-parabolic subgroup of $G$, see Definition \ref{de: positive hom} of $\ell$-positive homomorphism.
    \item \textbf{Lengths of highest weight vectors in fundumental representations:}\label{cond: comb fund} The class of functions of the form
    \begin{align*}
    \pi(g)\mapsto\sum_{i=1}^r c_i \log \|\varrho_i(gg_0)v_i\| + C,
    \end{align*}
    where $\varrho_i$ are $\ell$-fundamental representations, $g_0\in G(\ell)$, and $v_i$ are the appropriate highest weight vectors, which correspond to maximal $\ell$-parabolic groups, and are defined in \S \ref{sec:parabolic subgroups} for some choice of an 
    $\ell$-simple system, the $c_i$ are non-negative rational numbers, and $C\in\RR$.
    \item \textbf{Length of a single parabolic equivariant vector:} \label{part: maximal weight} \label{cond: length of single vector} The class of functions of the form
    \begin{align*}
    \pi(g)\mapsto \alpha \log \|\varrho(g)v\|,
    \end{align*}
    where $\varrho:G\to \GL(V)$ is an $\ell$-algebraic representations with a vector $v\in V(\ell)$ such that the ray $\ell v$ is stabilized by a parabolic subgroup, and $\alpha > 0$ is rational.
  \end{enumerate}
\end{thm}

\begin{proof}
The proof of the equivalence between the first three classes follows in a very similar way to the proof of Theorem \ref{thm: class equiv busemann}, using the more general theory which is presented in \S\ref{sec:parabolic subgroups}. 
The implication of first three classes from Part \ref{part: maximal weight} is similar to Theorem \ref{thm: f_v descends to Busemann}. 
We are left to show Part \ref{part: maximal weight} assuming the other three. 

Assume a function $f$ is in the third class, i.e., \[f(\pi(g)) = \sum_{i=1}^r c_i\log\| \varrho_i(gg_0)v_i\| + C,\] 
where $C\in\RR$, for any $1\le i\le d$, $c_i$ is rational and nonnegative, $g_0\in G$, and the vector $v_i$ is the previously chosen highest weight vector in $V_i(\ell)$. In particular, the line $\ell v_i$ is $P_i$-invariant for an $\ell$-maximal parabolic subgroup $P_i$. Set $P = \bigcap_{i}P_i$. Then, $P$ is a parabolic subgroup by Claim \ref{claim: parabolic expansion}. Applying $g_0$, the line $\ell \varrho_i(g_0)v_i$ is $g_0P_ig_0^{-1}$-invariant, and so $g_0Pg_0^{-1}$-invariant. 
Also, up to multiplication by an integer, we may assume that the $c_i$'s are integers. 

Let \[v' := v_1^{\otimes c_1}\otimes v_2^{\otimes c_2}\otimes \cdots \otimes v_r^{\otimes c_r} \in V_1^{\otimes c_1}\otimes V_2^{\otimes c_2}\otimes \cdots\otimes V_r^{\otimes c_r}\]
and $v=\varrho(g_0)v'$, where $\varrho=\varrho_1^{\otimes c_1}\otimes \varrho_2^{\otimes c_2}\otimes \cdots \otimes \varrho_r^{\otimes c_r}$. 
Then, the line $\ell v$ is $P$-invariant via the representation $\varrho$ and this vector in this representation represents the function as in class \ref{cond: length of single vector}. 

\end{proof}

\subsection{Kempf's result}
  The following definition is a normalized sense of how fast an algebraic torus shrinks a vector.
  
  \begin{de}[Shrinking rate of an algebraic torus]
  Let $\tau:\GG_m\to G$ be an algebraic nontrivial homomorphism, $\varrho:G\rightarrow\operatorname{GL}(V)$ be an algebraic representation, and $v\in V$. 
  
  $V$ can be decomposed into its eigenspaces with respect to the $\GG_m$ action $\varrho \circ \tau$
    \begin{equation}\label{eq: V eigen decomp}
        V = \bigoplus_{n\in \ZZ}V_n,
    \end{equation}
    where for any $n\in\ZZ$
    \begin{equation}\label{eq: eigenspace Vn}
      V_n:=\left\{w\in V:\forall t\in\GG_m,\,\varrho\circ\tau(t)w=t^n w\right\}.
    \end{equation}
    Let $m(v, \tau)$ be the maximal integer for which $v\in \bigoplus_{n \ge m(v, \tau)} V_n$. 
    
    Note that for every $n\in \NN$ we have \[m(v, \tau^n) = nm(v, \tau).\] 
    Using $e = \frac{\bd}{\bd s}\in T_1\GG_m$, we define a ``norm'' on $\Hom_\QQ(\GG_m,G)$ 
  \begin{equation}\label{eq: norm on Hom}
      \|\tau\|:= \sqrt{B(D_1\tau(e), D_1\tau(e))} = \sqrt{\sum_{n\in \ZZ}n^2 \dim V_n},
  \end{equation}
  where $B=B_e$ is the killing form (see \S\ref{sec: explicit construction}) and $D_1\tau:T_1\GG_m\to T_1 G\cong \fg$ is the differential of $\tau$ at $1$.
  \end{de} 
  \begin{remark}\label{rem: norm is integral}
    Note that for $\ta\in\fp$ which satisfies the conclusion of Claim \ref{claim: morphism is geodesic} for $\tau$, we have $\norm{\tau}=\norm{\ta}$, where the norm on $\fp$ is as defined in \S\ref{sec: explicit construction}. Moreover, it follows from \eqref{eq: norm on Hom} that $\norm {\ta}$ is a square root of an integer (see also the discussion in \cite[\S2]{kempf}). 
  \end{remark} 
    
\begin{lem}[{\cite[Lem. 3.2]{kempf}(c)}]\label{lem: alt defn m}
  The quantity $m(v, \tau)$ can be described as follows: Consider the map $f=f_{\tau,\varrho,v}:\GG_m\to V$ defined by \[
  f(t)=\varrho\circ\tau(t)v.\]
  For every functional $\varphi\in V^*$ we can consider the function $\varphi\circ f$ and compute its valuation at $t=0$. 
  $m(v, \tau)$ is the minimal such valuation.
\end{lem}
\begin{remark}
  The alternative description of $m(v, \tau)$ is the algebraic analogous for the shrinking rate of $\|\varrho(\tau(t))v\|$ as $t\to 0$. 
\end{remark}
   Recall the definition of a unstable vector from \S\ref{sec:shrink-rate function}. 
  \begin{thm}[{\cite[Thm. 4.2]{kempf}}]\label{thm: kempf1}
    Let $G$ be a semisimple $\ell$-algebraic group, $\varrho:G\to \operatorname{GL}(V)$ an algebraic $G$-representation, and $v\in V$ a unstable vector. 
    Then, there exists $\tau\in\Hom_\ell(\GG_m, G)$ which maximize $\frac{m(v, \tau)}{\|\tau\|}$ and for which $m(v, \tau)>0$. 
    This subgroup $\tau$ is unique up to taking power and conjugating by elements in $P_\tau$. 
    More explicitly, if $\tau'\in \Hom_\ell(\GG_m, G)$ are another such homomorphism that cannot be represented as nontrivial power, then there exists a unique element $u\in U_\tau$, the unipotent radical of $P_\tau$, such that $\tau' = u\tau u^{-1}$.
    We call such $\tau$ \emph{Kempf's homomorphism of $v$}.
  \end{thm} 
  The following is a corollary of the uniqueness in Theorem \ref{thm: kempf1}. 
  \begin{cor}\label{cor: kempf uniqu}
    In the setting of Theorem \ref{thm: kempf1},
    suppose that $k/\ell$ is a field extension. Then the extension of scalars of the Kempf's homomorphism of $v$ to $k$ is again a Kempf's homomorphism of the extended $v$.
  \end{cor}
The proof of Corollary \ref{cor: kempf uniqu} relies on the degeneracy of a certain Galois action, and so we need the following definitions. 

\begin{de}[Cocycle and cohomologous cocycles]\label{de: cocycles}
    Let $\ell$ be a field, $k$ be a finite extension of $\ell$, and $U$ be an $\ell$-algebraic group. Then, there is a Galois action  $\Gal(k/\ell)\acts U(k)$. A function $\alpha:\Gal(k/\ell)\to U(k)$ is called \emph{cocycle} if for every $\sigma_1, \sigma_2\in \Gal(k/\ell)$,
    \[\alpha(\sigma_1\sigma_2) = \alpha(\sigma_1)\sigma_1(\alpha(\sigma_2)).\]
    Two cocycles $\alpha, \beta$ are said to be \emph{cohomologous} if for some $u\in U$ we have 
    \[\forall \sigma\in \Gal(k/\ell),\quad \beta(\sigma) = u^{-1}\alpha(\sigma)\sigma(u).\]
\end{de}

We also need the following known result, which we prove for completion of the manuscript.

\begin{claim}\label{claim: cocycle}
    Assuming the notation of Definition \ref{de: cocycles}, every cocycle $\alpha:\Gal(k/\ell)\to \GG_a$ is cohomologous to the trivial cocycle. 
\end{claim}
\begin{proof}
    Note that here the group is additive. 
    Let $\alpha:\Gal(k/\ell)\to \GG_a$. 
    Let $x = \frac{1}{[k:\ell]}\sum_{\sigma\in \Gal(k/\ell)}\alpha(\sigma)$. 
    Then for every $\sigma'\in \Gal(k/\ell)$ we have 
    \begin{align*}
        x - \sigma'(x) & = \frac{1}{[k:\ell]}\sum_{\sigma\in \Gal(k/\ell)}(\alpha(\sigma) - \sigma'(\alpha(\sigma)))\\
        &= \frac{1}{[k:\ell]}\sum_{\sigma\in \Gal(k/\ell)}(\alpha(\sigma'\sigma) - \sigma'(\alpha(\sigma)))\\
        &= \frac{1}{[k:\ell]}\sum_{\sigma\in \Gal(k/\ell)}(\alpha(\sigma'))\\
        &= \alpha(\sigma').
    \end{align*}
    The claim follows.
\end{proof}
A simple Corollary of Claim \ref{claim: cocycle} is the following.
\begin{cor}\label{cor: cocycle}
    Suppose we have a pair of $\ell$ algebraic groups $U_1\rhd U_2$ such that $U_1/U_2 \cong \GG_a$ and $U_1(k)/U_2(k) \cong \GG_a(k)$.
    Then every cocycle $\alpha:\Gal(k/\ell)\to U_1(k)$ is cohomologous to a cocycle with values in $U_2(k)$. 
\end{cor}
\begin{proof}
    The composition of $\alpha$ and the projection $U_1(k)\to (U_1/U_2)(k) \cong \GG_a(k)$ is cohomologous to the trivial cocycle by Claim \ref{claim: cocycle}. Denote by $u\in (U_1/U_2)(k)$ the element such that $u^{-1}[\alpha(\sigma)]_{(U_1/U_2)(k)}\sigma(u) = 1$. Let $\tilde{u}\in U_1(k)$ be an element that projects to $u$. Then $\tilde u$ shows that $\alpha$ is cohomologous to a cocycle with values in $U_2(k)$. 
\end{proof}

  \begin{proof}[Proof of Corollary \ref{cor: kempf uniqu}]
    Let us assume that there is a counterexample to the claim, i.e., some 
    $\tau_1 \in \Hom_k(\GG_m, G)$ such that 
    \[\frac{m(v, \tau_1)}{\|\tau_1\|} > \frac{m(v, \tau)}{\|\tau\|}.\]
    Then $\tau_1$ is defined over a finitely generated algebra $k_1/\ell$. 
    In particular, $k_1$ has a quotient $k_2$ that is a finite extension of $\ell$ and the projection of scalars $\tau_2$ of $\tau_1$ to $Hom_{k_2}(\GG_m, G)$ also has 
    \[\frac{m(v, \tau_2)}{\|\tau_2\|} > \frac{m(v, \tau)}{\|\tau\|}.\]
    Hence we may assume without loss of generality that $k$ is a finite field extension of $\ell$. 
    We may replace $k$ by its Galois closure $f$ over $\ell$, applying the argument twice for $f/\ell$ and $f/k$. 
    Thus, we may assume that $k$ is Galois over $\ell$. 

    For every element $\sigma\in \Gal(k/\ell)$, the two homomorphism $\tau_1, \sigma(\tau_1)$ are both Kempf's homomorphism of $v$. The uniqueness of the parabolic in Theorem \ref{thm: kempf1} asserts that $P_{\tau_1} = P_{\sigma(\tau_1)}$. In particular, the parabolic subgroup $P_{\tau_1}$ is defined over $\ell$. 
    In addition, this uniqueness also shows that for every $\sigma$ there is a unique $u_\sigma\in U_{\tau_1}(k)$ such that $u_\sigma^{-1}\tau_1 u_{\sigma} = \sigma(\tau_1)$.
    Applying the last for $\sigma_1 ,\sigma_2\in \Gal(k/\ell)$ we obtain  
    \begin{align*}
        u_{\sigma_1\sigma_2}^{-1}\tau_1 u_{\sigma_1\sigma_2} &= \sigma_1\sigma_2(\tau_1) = \sigma_1(u_{\sigma_2}^{-1}\tau_1 u_{\sigma_2}) \\
        &= \sigma_1(u_{\sigma_2})^{-1}\sigma_1(\tau_1)\sigma_1( u_{\sigma_2})\\& = (u_{\sigma_1}\sigma_1(u_{\sigma_2}))^{-1} \tau_1u_{\sigma_1}\sigma_1(u_{\sigma_2}).
    \end{align*}
    Thus, the uniqueness of $u_{\sigma_1\sigma_2}$ implies $u_{\sigma_1\sigma_2} = u_{\sigma_1}\sigma_1(u_{\sigma_2})$. 
    Hence, $$u_\bullet:\Gal(k/\ell)\to U_{\tau_1}(k)$$ is a cocycle.

    Next, we claim that if $u_\bullet$ is cohomologous to the trivial cocycle, that is, that there is $u\in U_{\tau_1}(k)$ such that $u_\sigma = u^{-1}\sigma(u)$, then we done. 
    Indeed, this imply that $u\tau_1u^{-1}$ is invariant under the Galois action, 
    \begin{align*}
        \sigma(u\tau_1u^{-1})& = \sigma(u)u_\sigma^{-1}\tau_1 u_{\sigma}\sigma(u)^{-1}\\
        &= \sigma(u)(u^{-1}\sigma(u)) ^{-1}\tau_1 u^{-1}\sigma(u)\sigma(u)^{-1} = u\tau_1u^{-1},
    \end{align*}
    and hence $u\tau_1u^{-1}$ is defined over $\ell$, a contradiction to the maximality of $\tau$.

    Let us now show that $u_\bullet$ is cohomologous to the trivial cocycle. Since $U_{\tau_1}$ is unipotent, it has an $\ell$-algebraic composition series $U_{\tau_1} = U_0\rhd U_1\rhd\dots \rhd U_s$ such that $U_i/U_{i+1}\cong \GG_a$ and $U_i(k)/U_{i+1}(k)\cong \GG_a(k)$. 
    Then, iteratively replace the cocycle $u_\bullet$ with cohomologous cocycles with values in $U_i$ using Corollary \ref{cor: cocycle}, and finally show that $u_\bullet$ is cohomologous to the trivial cocycle. 
  \end{proof}
  Recall the definition of the shrink-rate function of a vector from \S\ref{sec:shrink-rate function}, as well as the definitions of the fastest shrinking geodesic of a function and the shrinking rate of it from Definition \ref{de: fastest shrinking}. 
  
  The following claim identifies the fastest shrinking geodesic of the shrink-rate function of a vector and its Kempf's homomorphism.

  \begin{lem}\label{lem: kempf is fastest}
    Let $\varrho:G\to \GL(V)$ be an $\ell$-algebraic representation, and $v\in V(\ell)$ be an unstable vector. Then, the fastest shrinking geodesic of $f_v$ (as defined in \S\ref{sec:shrink-rate function}) is of the form 
    $\gamma_\tau$ where $\tau\in \Hom_\ell(\GG_m, G)$ is Kempf's homomorphism of $v$.
    In addition, the shrinking rate satisfies $a_{f_v, M, o} = \frac{m(v,\tau)}{\|\tau\|}$.
  \end{lem}
  \begin{proof}
    Let $\tau \in\Hom_\ell(\GG_m, G)$ be Kempf's homomorphism of $v$. 
    By Corollary \ref{cor: kempf uniqu}, the extension of scalars of $\tau$ to $\RR$ is also Kempf's homomorphism of $v$.
    Thus extension of scalars to $\RR$ does not change the claim, and we may assume that $\ell = \RR$. 

    Let $\varsigma\in\Hom_\RR(\GG_m,G)$ be nontrivial.
    Then, by Lemma \ref{lem: alt defn m} and the definitions of $f_v$, we have \[\frac{m(v, \varsigma)}{\|\varsigma\|} = \lim_{s\to \infty}\frac{-f_v(\gamma_\varsigma(s))}{s}.\]
    Using Claim \ref{claim: morphism is geodesic}, this implies that $\tau$ is the algebraic geodesic rays that shrinks $f_v$ the fastest. 
    In particular, $a_{f_v, M, o} \ge \frac{m(v,\varsigma)}{\|\varsigma\|}$
    for all $\varsigma\in\Hom(\GG_m, G)$ and equality occurs if and only if the fastest shrinking geodesic is algebraic.
    
    Without loss of generality, assume that the fastest shrinking geodesic is $\gamma(s) = \exp(s\ta)$ for some $\mathtt {a}\in \fa$. 
    Note that the set of algebraic geodesic rays is dense in the set of all rays, and that by Observation \ref{obs: convexity of f_v} and Claim \ref{claim: lim continuous}, the function \[
    \gamma\mapsto\lim_{s\to \infty}\frac{-f_v(\gamma(s))}{s},\] 
    defined for all geodesic rays $\gamma$ in $\pi(A)$, is continuous. 
    It follows that $a_{f_v, M, o}$ is at most the supremum of 
    $\lim_{s\to \infty}\frac{-f_v(\gamma(s))}{s}$ over all algebraic geodesic rays $\gamma$ from $o$ in $\pi(A)$. This concludes the proof.
  \end{proof}

\section{Alternative proof of Theorem \ref{thm: NimishPengyu}}\label{sec: proof of main theorem}

Let $\varrho:G\to \GL(V)$ be a $\QQ$-representation, $V$ be equipped with a $K$-invariant norm, and $v$ be an unstable non-zero vector in $V(\QQ)$.

Define the shrink-rate function $f_v:G\to \RR$ of $v$ as in \S\ref{sec:shrink-rate function}, i.e., \[f_v(\pi(g)) = \log\|gv\|.\]
Note that the function $f_v$ is the quantity on the left hand side of \eqref{ineq: more then repfunc2}, which we wish to control. 
By Theorem \ref{thm: busemann bound} there exists a fastest shrinking geodesic for $f_v$ (see \S\ref{sec:The fastest shrinking geodesic} for the definition of a fastest shrinking geodesic), denote it by $\gamma$, and $C\in \RR$ such that for every $x\in M:=K\backslash G$, we have
\begin{align}\label{ineq: f more then busemann intro}
f_v(x) \ge \tilde{a}\beta_\gamma(x) + C,
\end{align}
where $\beta_\gamma$ is the Busemann function associated to $\gamma$ (defined in \S\ref{sec:Busemann functions}) and $\tilde a$ is the shrinking rate of $f_v$ on $\gamma$, that is, $\tilde{a} = -\lim_{t\to \infty}\frac{f_v(\gamma(t))}{t}$. Moreover, since $v\in V(\QQ)$, by Lemma \ref{lem: kempf is fastest} $\gamma$ is $\QQ$-algebraic, and hence $\tilde{a}\beta_\gamma$ is a $\QQ$-renormalized Busemann function (see Definition \ref{de: k-algebraic}).

Now, Theorem \ref{thm: class equiv busemann alg} connects the above Busemann function to a $\QQ$-highest weight representation, $\varrho':G\rightarrow\GL(W)$ implying that there exist $w\in W(\QQ)$ stabilized by a parabolic subgroup and non-negative $\tilde{a}$ and $\tilde{C}$ such that for any $g\in G$
\begin{equation}\label{eq: beta comb of fund}
    \beta_\gamma(\pi(g))= \tilde\alpha \log \|\varrho'(g)w\| - \tilde{C}.
\end{equation}
Note that $w$ is the highest weight vector with respect to some maximal $\ell$-split torus and a choice of a simple system by Lemma \ref{lem:stabilized by parab}.


Combining \eqref{ineq: f more then busemann intro} and \eqref{eq: beta comb of fund}, one may deduce that for some non-negative $a,c$ and all $g\in G$, we have \[
f_v(\pi(g)) \ge a \log \|\varrho'(g)w\| - c.\]
The claim now follows from the definition of $f_v$.

\appendix

\section{Proof of Lemma \ref{lem: construction of bilinear form}}
Here we prove Lemma \ref{lem: construction of bilinear form}. We restate it below for the convenience of the reader.  

\begin{lem}[Construction of bilinear form]\label{lem: appendix}
Let $\varrho:G\rightarrow\operatorname{GL}(V)$ be an $\RR$-representation. Then, there is a $K$-invariant positive bilinear form $\bra \cdot, \cdot\ket$ on $V$ so that the linear spaces $V_\lambda$ are orthogonal with respect to it (see \S\ref{sec:representations} for the definition of $V_\lambda$). 
\end{lem}

For the proof we need some definitions regarding Galois group actions. 

\begin{de}[Galois Action]\label{de: Galois}
Denote by $\Gal(\CC/\RR)$ the Galois group, i.e., the group of automorphisms of $\CC$ which fix $\RR$ pointwise. Then, the only elements in $\Gal(\CC/\RR)$ are the identity map and the complex conjugation, which we denote by $\conjj$. Moreover, there is an action of $\Gal(\CC/\RR)$ on:
\begin{enumerate}
    \item[(i)] \textbf{complexifications of $\RR$-algebraic objects}, such as $G^\CC$ and $V\otimes \CC$, in a natural way.
    \item[(ii)] \textbf{the category of complex vector spaces} in the following way: 
    Given a vector space $U$, let $U^{\conjj}$ be a vector space with the same set of points as $U$, let the addition on $U^{\conjj}$ be the same as on $U$, and the multiplication by a scalar $c\in \CC$ on $U^{\conjj}$ be the multiplication by the conjugate of $c$ on $U$. We denote the topological map from $U$ to $U^{\conjj}$ by $\conjj$.
    \item[(iii)] \textbf{the category of complex representations of $G^\CC$} as follows: 
    For $\sigma:G^\CC\to \GL(U)$ we can set a (canonical) representation $\sigma^{\conjj}:G^{\CC}\to \GL(U^{\conjj})$ which acts by \[
    \sigma^{\conjj}(g)(u) = \conjj(\sigma({\conjj}(g))({\conjj(u)})).\]
    \item[(iv)] \label{def: action on action} \textbf{$G$ actions} have the following phenomenon:
    Assume $G$ is a group acting on a space $X$ via $\sigma:G\to \operatorname{Aut}(X)$. 
    An action $\Gal(\CC/\RR) \acts \varrho$ is a tuple of actions $\Gal(\CC/\RR)\acts G, \Gal(\CC/\RR)\acts X$  such that \[\conjj(\sigma(g)(x)) = \sigma(\conjj(g))(\conjj(x)),\] 
    for every $g\in G, x\in X$. 
    The above is equivalent to an extension of $\sigma$ to an action of the semidirect product $\tilde{\sigma}: \Gal(\CC/\RR) \ltimes G\to \Aut(X)$.
\end{enumerate}
\end{de}
The main tool we use in the proof is the unitarian trick:
\begin{thm}[{Unitarian Trick, \cite[Thm. 4.11.14]{vara}}]\label{thm: unitrick}
    Let $H$ be a semisimple simply-connected complex Lie group with a maximal compact subgroup $K_H$. 
    The restriction of $H$-representations to $K_H$-representations induces a bijection from the category of complex algebraic representations of $H$ to the category of complex finite dimensional representations of $K_H$. 
\end{thm}
If $G$ is not algebraically simply-connected, replace it with its algebraic simply-connected cover. This preserves the conclusion of Lemma \ref{lem: appendix}.  

To prove Lemma \ref{lem: appendix}, we study real representations of a real $G$ by studying the complex algebraic represenations of the complexification $G^\CC$. 
Then, we will Unitatian trick to relate the discussion to study of complex representation of a maxiaml compact group $K^\CC\subset G^\CC$. 
The compactness of $K^\CC$ will make it easy to find the desired positive definite quadratic form.

\begin{proof}[Proof of Lemma \ref{lem: appendix}]
Let $G$ be a real algebraic semisimple Lie group, and $\varrho:G\to \GL(V)$ a representation. 
Since $G$ is an algebraic group, we can consider its complexification $G^\CC$, with the conjugation action $\conjj_G:G^\CC\to G^\CC$. 
\begin{claim}
    There is a maximal compact subgroup $K^{\CC}\subset G^\CC$ such that $\conjj_G(K^\CC) = K^\CC$, and $K^\CC\cap G = K$.
\end{claim}
\begin{proof}
    To find an explicit description of $K^{\CC}$, the maximal compact subgroup of $G^\CC$, we use standard constructions, see \cite[\S VI,2]{knapp}. Recall that $\theta:\fg\to\fg$ is the Cartan involution antihomomorphism, it acts on the Lie algebra $\fg = \fk \oplus \fp$ by inverting $\fk$ and preserving $\fp$.
    Since $\fg^\CC = \fg\otimes \CC$, we may extend $\theta$ to a real involution $\theta^\CC: \fg^\CC\to \fg^\CC$ by defining \[\theta^{\CC} = \theta \otimes \conjj.\]
    Since $\fg\otimes \CC = \fk\otimes\CC\oplus \fp \otimes \CC$, the above defines a decomposition $\fg^\CC=\fk^\CC\oplus\fp^\CC$, viewed as real vector spaces, for
    \begin{align*}
    \fk^\CC := & \fk \oplus i\fp,\\
    \fp^\CC := & i\fk\oplus \fp.
    \end{align*}
    Moreover, $\fk^\CC$ is the $1$ eigenspace of $\theta^\CC$ and $\fp^\CC$ is the $-1$ eigenspace.
    
    Since $\theta^\CC$ is a Lie algebra homomorphism, we deduce that its fixed points, i.e., $\fk^\CC$, is also a Lie algebra.
    Direct computation shows that the killing form of the group $G^\CC$, viewed as a real algebraic group, is positive on $\fp^\CC$ and negative on $\fk^\CC$.
    This implies that this is indeed the cartan decomposition of $G^\CC$, thought of as a real Lie group.
    In particular, $K^\CC:= \exp \fk^\CC$ is a maximal compact subgroup of $G^\CC$. 
    Since $\fk^\CC = \fk \oplus i\fp$ is invariant to complex conjugations, we deudce tht $\conjj_G(K^\CC) = K^\CC$, and since 
    $\fk\subseteq \fk^\CC$ we deudce that $K\subseteq K^\CC$.
\end{proof}

We assume without loss of generality that $\varrho$ is irreducible.
Then, one of the following holds:
\begin{enumerate}
    \item \label{case: irreducible} $\varrho \otimes \CC$ is an irreducible representation of $G^\CC$.
    \item \label{case: reducible} There is a decomposition $\varrho\otimes \CC \cong \sigma\oplus \sigma^{\conjj}$.
\end{enumerate}

We will first prove the lemma for $\varrho$ as in the case \eqref{case: irreducible}. 
In that case, $\varrho|_{K^\CC}$ is also irreducible by Theorem \ref{thm: unitrick}. 
\begin{claim}\label{claim: integration}
On every irreducible complex $K^\CC$-representation there is exactly one $K^\CC$ invariant, positive definite Hermitian form up to multiplication by a positive scalar.
\end{claim}
\begin{proof}
Let $\sigma:K^\CC\to \GL(U)$ be an irreducible representation.
There is at least one such invariant positive definite Hermitian form, as we can average a non-invariant positive definite Hermitian form along the $K^\CC$ action.

If there are two such positive definite Hermitian forms, $H_1, H_2$, consider the infimum 
\[\alpha:=\inf\{\alpha'>0: H_1-\alpha' H_2 \text{ is not positive definite}\}.\]
Then, the Hermitian form $H_1-\alpha H_2$ is non-negative definite and $K^\CC$-invariant. Moreover, the set \[
W = \{w\in U:(H_1-\alpha H_2)(w,w)=0\}\] is nonempty, not equal to $U$, and is $K^\CC$-invariant. This contradicts the irreducibility of $\sigma$ as a $K^\CC$ representation. 
\end{proof}

Denote the unique invariant, positive definite Hermitian form on $V^{\CC}=V\otimes \CC$ by $H^\CC$.
Since the the Galois group $\Gal(\CC/\RR)$ acts on the $K^\CC$, $V^\CC$,  and the $K^\CC$-action $\varrho \otimes \CC$ (i.e., satisfies property \ref{def: action on action} of Definition \ref{de: Galois}), we deduce that $H^\CC$ is Galois invariant. 
Hence, $H^\CC$ is induced from a positive definite symmetric billinear form $H$ on $V$. 

We are left to show that $H$ satisfies the desired property, that is, the linear spaces $V_\lambda$ are orthogonal with respect to it. 
Assume that $V\cong \RR^n$, $V\otimes \CC\cong \CC^n$, and $H^\CC$ is the standard Hermitian form on $\CC^n$. Then, it is enough to show that $A$ is sent by $\varrho^\CC$ to a group of Hermitian matrices.

As in Definition \ref{de: Galois}, the representation $\varrho$ defines a map $\varrho^\CC:G^\CC\to \GL_n(\CC)$.
Define $\theta_n:\GL_n(\CC)\to \GL_n(\CC)$ by \[\theta_n(M) = \conjj(M^{-t}).\] 
Since $\varrho^\CC(A)$ is invariant under $\theta^\CC$, in order to show that $\varrho^\CC(A)$ is also invariant under $\theta_n$, it suffices to prove that 
$\theta_n\circ \varrho^\CC \circ \theta^\CC = \varrho^\CC$. 
Since $K^\CC$ preserves the Hermitian form $H^\CC$, it follows that $\varrho^\CC(K_\CC)\subseteq U(n)$. In particular, $\theta_n\circ \varrho^\CC \circ \theta^\CC|_{K^\CC} = \rho|_{K^\CC}$. 
Since $K^\CC$ is Zariski dense in $G^\CC$, it follows that $\theta_n\circ \varrho^\CC \circ \theta^\CC = \varrho^\CC$, as desired.

Now, assume we are in case \eqref{case: reducible}, i.e., $\varrho\otimes \CC \cong \sigma\oplus \sigma^{\conjj}$,
for some $\sigma:G^\CC\to \GL(U)$.
The $\RR$-linear map $\conjj$ on $V\otimes \CC \cong U\oplus U^{\conjj}$ defines an action on the space $\operatorname{Pos}(V\otimes \CC)$ of positive definite Hermitian forms $H$ on $V\otimes \CC$ by \[
\conjj(H)(v_1,v_2) = \conjj(H(\conjj(v_1), \conjj(v_2))).\]
The group $G^\CC$ acts on $\operatorname{Pos}(V\otimes \CC)$ 
as well, and the $K^\CC$ invariant Hermitian forms 
\[\operatorname{Pos}(V\otimes \CC)^{K^{\CC}} = \operatorname{Pos}(U)^{K^{\CC}}\oplus \operatorname{Pos}(U^{\conjj})^{K^{\CC}}.\]
Each of these components is isomorphic to $\RR^{>0}$ by Claim \ref{claim: integration}. 
Note that $\Gal(\CC/\RR)$ acts on the action $K^\CC \acts \operatorname{Pos}(V\otimes \CC)$, and when restricting this action to $\operatorname{Pos}(V\otimes \CC)^{K^{\CC}}$ we get that $\conjj$ replaces the two isomorphic components. 
Adding the conjugation to the acting group, we get a semidirect product,
$\Gal(\CC/\RR)\ltimes K^{\CC}\acts \operatorname{Pos}(V\otimes \CC)$. The space of invariants of this action is $\RR^{>0}$, and hence up to multiplication by positive scalar there is again only one such positive quadratic form. 
Since $\conjj$ preserves $K^{\CC}$ in $G^{\CC}$ we get an action of the compact group $\Gal(\CC/\RR)\ltimes K^{\CC}$ on $V^\CC$. 
The rest of the proof is similar to case (1).
\end{proof}

The next example shows that case (2) in the proof of Lemma \ref{lem: appendix} can occur. 
\begin{ex}
Let $G=\operatorname{SU}(3)$, $V=\CC^3$, and $\varrho:G\rightarrow\operatorname{GL}(V)$ be the standard inclusion representation.
Since $G$ acts transitively on the unit sphere in $V$, it is irreducible. Moreover, restriction of scalars $\varrho'=\res_{\CC/\RR}(\varrho)\in {\rm Rep}_{\RR}(G)$ is also irreducible. 
However, the complexification of $\varrho'$ is reducible, as $\varrho'\otimes_{\RR} \CC = \varrho\otimes_\RR\CC$ is of complex dimension $6$ and has the $\CC$-linear multiplication map to $\varrho = \varrho\otimes_\CC\CC$. 
\end{ex}


\section*{Funding statement}
The first autor is supported by ERC grant HomDyn, ID 833423. The second author did not receive support from any organization for the submitted work.

\end{document}